\documentclass{amsart} 
\usepackage{graphicx}
\usepackage{epsfig}

\usepackage{amssymb,amsmath,amsthm,amsfonts}

\usepackage[letterpaper, margin=1.2in]{geometry} 

\usepackage {latexsym,enumerate}
\usepackage{bbm}

\allowdisplaybreaks
\newcommand{\nc}{\newcommand}
\nc{\les}{\lesssim}
\nc{\nit}{\noindent}
\nc{\nn}{\nonumber}
\nc{\D}{\partial}
\nc{\diff}[2]{\frac{d #1}{d #2}}
\nc{\diffn}[3]{\frac{d^{#3} #1}{d {#2}^{#3}}}
\nc{\pdiff}[2]{\frac{\partial #1}{\partial #2}}
\nc{\pdiffn}[3]{\frac{\partial^{#3} #1}{\partial{#2}^{#3}}}
\nc{\abs}[1] {\lvert #1 \rvert}
\nc{\cAc}{{\cal A}_c}
\nc{\cE}{{\cal E}}
\nc{\cF}{{\cal F}}
\nc{\cP}{{\cal P}}
\nc{\cV}{{\cal V}}
\nc{\cQ}{{\cal Q}}
\nc{\cGin}{{\cal G}_{\rm in}}
\nc{\cGout}{{\cal G}_{\rm out}}
\nc{\cO}{{\cal O}}
\nc{\Lav}{{\cal L}_{\rm av}}
\nc{\cL}{{\cal L}}
\nc{\cB}{{\cal B}}
\nc{\cZ}{{\cal Z}}
\nc{\cR}{{\cal R}}
\nc{\cT}{{\cal T}}
\nc{\cY}{{\cal Y}}
\nc{\cX}{{\cal X}}
\nc{\cXT}{{{\cal X}(T)}}
\nc{\cBT}{{{\cal B}(T)}}
\nc{\vD}{{\vec \mathcal{D}}}
\nc{\efield}{\mathcal{E}}
\nc{\vE}{{\vec \efield}}
\nc{\vB}{{\vec \mathcal{B}}}
\nc{\vH}{{\vec \mathcal{H}}}
\nc{\mR}{\mathcal R}
\nc{\mF}{\mathcal F}
\nc{\ty}{{\tilde y}}
\nc{\tu}{{\tilde u}}
\nc{\tV}{{\tilde V}}
\nc{\Pc}{{\bf P_c}}
\nc{\bx}{{\bf x}}
\nc{\bX}{{\bf X}}
\nc{\bXYZ}{{\bf XYZ}}
\nc{\bY}{{\bf Y}}
\nc{\bF}{{\bf F}}
\nc{\bS}{{\bf S}}
\nc{\dV}{{\delta V}}
\nc{\dE}{{\delta E}}
\nc{\TT}{{\Theta}}
\nc{\dPsi}{{\delta\Psi}}
\nc{\order}{{\cal O}}
\nc{\Rout}{R_{\rm out}}
\nc{\eplus}{e_+}
\nc{\eminus}{e_-}
\nc{\epm}{e_\pm}
\nc{\sgn}{\text{sgn}}
\nc{\eps}{\varepsilon}
\nc{\vnabla}{{\vec\nabla}}
\nc{\G}{\Gamma}
\nc{\w}{\omega}
\nc{\mh}{h}
\nc{\mg}{g}
\nc{\vphi}{\varphi}
\nc{\tlambda}{\tilde\lambda}
\nc{\be}{\begin{equation}}
	\nc{\ee}{\end{equation}}
\nc{\ba}{\begin{eqnarray}}
	\nc{\ea}{\end{eqnarray}}

\nc{\g}{\gamma}
\nc{\ol}{\overline}

\newtheorem{theorem}{Theorem}[section]
\newtheorem{lemma}[theorem]{Lemma}
\newtheorem{prop}[theorem]{Proposition}
\newtheorem{corollary}[theorem]{Corollary}

\newtheorem{rmk}[theorem]{Remark}

\nc{\pT}{\partial_T}
\nc{\pz}{\partial_z}
\nc{\pt}{\partial_t}
\nc{\la}{\langle}
\nc{\ra}{\rangle}
\nc{\infint}{\int_{-\infty}^{\infty}}
\nc{\halfwidth}{6.5cm}
\nc{\figwidth}{10cm}
\newcommand{\f}{\frac}

\nc{\nlayers}{L} \nc{\nsectors}{M}
\nc{\indicator}{\mathbf{1}}
\nc{\Rhole}{R_{\rm hole}}
\nc{\Rring}{R_{\rm ring}}
\nc{\neff}{n_{\rm eff}}
\nc{\Frem}{F_{\rm rem}}
\nc{\R}{\mathbb R}
\nc{\mJ}{\mathcal J}
\nc{\C}{\mathbb C}
\nc{\Z}{\mathbb Z}
\nc{\N}{\mathbb N}
\nc{\DD}{\Delta}
\nc{\cD}{\mathcal D}
\nc{\lnorm}{\left\|}
\nc{\rnorm}{\right\|}
\nc{\rnormp}{\right\|_{\ell^{p,\eps}}}
\nc{\rar}{\rightarrow} 
\newcommand{\Kato}{{\mathcal K}}

\newcommand{\U}{{\mathcal U}}
\newcommand{\B}{{\mathcal B}}
\def\norm[#1][#2]{\|#1\|_{#2}}

\newtheorem{proposition}[theorem]{Proposition}

\theoremstyle{remark}

\begin{document}
	
	\begin{abstract}
		
		We prove a family of dispersive estimates for the higher order Schr\"odinger equation $iu_t=(-\Delta)^mu +Vu$ in $n$ spatial dimensions, for $m\in \mathbb N$ with $m>1$ and $2m<n<4m$.  Here $V$ is a real-valued potential belonging to the closure of $C_0$ functions with respect to the generalized Kato norm, which has critical scaling.  Under standard assumptions on the spectrum, we show that $e^{-itH}P_{ac}(H)$ satisfies a $|t|^{-\frac{n}{2m}}$ bound mapping $L^1$ to $L^\infty$ by adapting a Wiener inversion theorem.  We further show the lack of positive resonances for the operator $(-\Delta)^m +V$ and a family of dispersive estimates for operators of the form $|H|^{\beta-\frac{n}{2m}}e^{-itH}P_{ac}(H)$ for $0<\beta\leq \frac{n}{2}$. The results apply in both even and odd dimensions in the allowed range.
		
	\end{abstract}

	\title[Dispersive estimates for higher order  Schr\"odinger operators]{\textit{Dispersive estimates for higher order Schr\"odinger operators with scaling-critical potentials} } 
	
	\author[Erdo\smash{\u{g}}an, Goldberg, Green]{M. Burak Erdo\smash{\u{g}}an, Michael Goldberg and William~R. Green}
	\thanks{  The first author was partially supported by the NSF grant  DMS-2154031 and Simons Foundation Grant 634269.  The second author is partially supported by Simons Foundation
		Grant 635369. The third author is partially supported by Simons Foundation
		Grant 511825. }
	\address{Department of Mathematics \\
		University of Illinois \\
		Urbana, IL 61801, U.S.A.}
	\email{berdogan@illinois.edu}
	\address{Department of Mathematics\\
		University of Cincinnati \\
		Cincinnati, OH 45221 U.S.A.}
	\email{goldbeml@ucmail.uc.edu}
	\address{Department of Mathematics\\
		Rose-Hulman Institute of Technology \\
		Terre Haute, IN 47803, U.S.A.}
	\email{green@rose-hulman.edu}
	%\subjclass{35Q41, 42B20}

	\maketitle

	\section{Introduction}
	
	We consider the higher order Schr\"odinger equation 
	\begin{align*}
		i\psi_t =(-\Delta)^m\psi +V\psi, \qquad x\in \R^n,  \quad  m\in \mathbb N,
	\end{align*}
	with a real-valued, decaying potential  $V$ in $n$ spatial dimensions when $2m<n<4m$.
	
	We denote the free higher order Schr\"odinger operator as $H_0=(-\Delta)^m$ and the perturbed operator by  $H=(-\Delta)^m+V(x)$.
	The free solution operator $e^{-itH_0}$ satisfies the dispersive estimate
	$$
	\|e^{-it H_0}\|_{1\to \infty} \leq C_{n,m} |t|^{-\frac{n}{2m}}.
	$$ 
	We show, with $P_{ac}(H)$ the projection onto the absolutely continuous subspace of $H$, that the solution operator $e^{-itH}P_{ac}(H)$ satisfies the same bound for a class of scaling-critical potentials.  That is, due to commutation relations between powers of the Laplacian and dilations, perturbing $H_0$ by a re-scaled potential of the form $V_r(x)=r^{2m} V(rx)$ obeys the same dispersive bounds as perturbing by $V$.  The admissible class of potentials in our argument is invariant under the scaling from $V$ to $V_r$.
	
	A natural way to capture scaling properties is through the use of Kato-type norms.	For $0 \leq \alpha \leq n$, we define the $\alpha$-Kato norm in $\R^n$ to be
	\begin{align} \label{eq:KatoNorm}
		\|V\|_{\Kato^\alpha} =\sup_{y\in \mathbb R^n} \int_{\R^n} \frac{|V(x)|}{|x-y|^\alpha}\, dx.
	\end{align}
	When $\alpha = 0$ this is the $L^1(\R^n)$ norm. The Kato norm with $\alpha = n-2m$ has critical scaling with respect to the higher order Laplacian in the sense above, and naturally arises in that $V(-\Delta)^{-m}$ is a bounded operator on $L^1(\R^n)$ precisely when $V \in \Kato^{n-2m}$. In our main results we consider potentials in the closure of $C_0$ with respect to the $\Kato^{n-2m}$ norm and we denote this subspace $\Kato$.  
	
	To state our main theorem, we recall that a resonance at energy $\lambda$ is a distributional solution to $H\psi=\lambda \psi$ with $\psi\notin L^2(\R^n)$, but in a slightly larger space depending on the dimension and order of the Laplacian.  For $H=(-\Delta)^m+V$, a resonance occurs when $(1+|x|)^{-\sigma}\psi \in L^2$ for some $\sigma> 2m-\frac{n}{2}$.  It may similarly be characterized in terms of non-invertibility of certain operators related to the free resolvent.	
	Here, and throughout the paper, we write $A\les B$ to mean there exists a constant $C>0$ so that $A\leq CB$.  Our main result is
	
	\begin{theorem}\label{thm:main}
		
		Let $m\in \mathbb N$ be such that $2m<n<4m$. If $V\in \Kato$, the closure of $C_0$ functions in the $\Kato^{n-2m}$ norm, is real-valued such that the operator $H=(-\Delta)^m+V$ has no eigenvalues on $[0,\infty)$ and no resonance at zero energy, then
		$$
		\|e^{-itH} P_{ac}(H)\|_{L^1\to L^\infty} \les |t|^{-\frac{n}{2m}}.
		$$
		
	\end{theorem}
	In particular as part of the proof of Theorem~\ref{thm:main} we obtain the fact that the operator $H$ has no resonances in $(0,\infty)$ under the assumption that there are no eigenvalues in $(0,\infty)$, see  Corollary~\ref{cor:no res} below.
	
	To motivate our second main result we note that the Fourier transform of the function $e^{i|\xi|^{2m}}|\xi|^{2m\alpha}$ on $\R^n$ is bounded  provided that  $-\frac{n}{2m}< \alpha\leq \frac{n}2-\frac{n}{2m}$; see  Lemma~\ref{lem:FT nl3m} below for a proof of this when $n=1$. By scaling and letting $\beta=\alpha+\frac{n}{2m}$, this implies the following family of dispersive estimates for $H_0=(-\Delta)^m$: 
	$$
	\||H_0|^{\beta-\frac{n}{2m}}e^{-itH_0}\|_{L^1\to L^\infty} \les |t|^{-\beta }, \,\,\,\,0<\beta\leq\frac{n}2.
	$$ 
	Dispersive estimates of this type (with $\beta=\frac{n}2$) were used by the authors in \cite{EGGcounter} in dimensions $n\geq 4m$ to construct counterexamples to the $L^p$ boundedness of the wave operators for insufficiently smooth potentials.  The obstruction in~\cite{EGGcounter} is not present in lower dimensions. 
	Here we prove 
	
	\begin{theorem}\label{thm:main2}
		
		Let $m\in \mathbb N$ be such that $2m<n<4m$.
		If $V\in \Kato$ is real-valued such that the operator $H=(-\Delta)^m+V$ has no eigenvalues on $[0,\infty)$ and no resonance at zero energy, then
		$$
		\||H |^{\beta-\frac{n}{2m}} e^{-itH} P_{ac}(H)\|_{L^1\to L^\infty} \les |t|^{-\beta },  
		$$
		provided that $0<\beta  \leq \frac{n}2$.
	\end{theorem}
	Note that Theorem~\ref{thm:main} is a special case of this when $\beta=\frac{n}{2m}\in (1,2)$.

	Dispersive estimates in the case of $m=1$ are well-studied in all dimensions.  See, for example, the recent survey paper \cite{SchlagSurvey}, and the references therein.
	When $m=1$ and $n=3$, dispersive bounds for scaling critical potentials were proven by Beceanu and the second author, \cite{BG}, by applying an operator-valued Wiener inversion theorem.  This type of inversion was first used by Beceanu in \cite{Bec}.  The operator-valued Wiener inversion theorem was later extended by Beceanu and the second author in \cite{BGwave}  to study a three-dimensional wave equation, and by Beceanu in \cite{BecWave} to study the wave operators in three dimensions.  Bui, Duong and Hong employed the Wiener inversion results in a further study of a three-dimensional wave equation.  Beceanu and Schlag further refined results on the structure of the wave operators in three dimensions, \cite{BS,BS2}.  Hill used the Wiener inversion method to study the case of $n=1$ and $m=1,2$, \cite{Hill}. We note that this is the first time this method has been used in even dimensions, or on higher order operators for scaling-critical potentials. We extend the Wiener inversion method in several directions to obtain these results. Namely, we utilize mappings between more generalized Kato spaces to accommodate the more complicated resolvent expansions in both even and odd dimensions for the polyharmonic case $m>1$. This also allows us to control higher order derivatives of the resolvents and to develop the necessary spectral theory.

	The literature on dispersive estimates for $m>1$ has blossomed recently.  First, local dispersive bounds were considered as maps between weighted $L^2$ spaces.  Feng, Soffer and Yao proved local dispersive bounds for fourth order Schr\"odinger operators, when $m=2$ and $n=3$ or $n>4$ in \cite{fsy}.  The third author and Toprak considered global, $L^1\to L^\infty$ estimates in the case of $n=4$, and the case of $n=3$ with the first author, \cite{GT4,egt}.  Global bounds were proven by Soffer, Wu and Yao when $n=1$, \cite{SWY}, and by Li, Soffer and Yao when $n=2$, \cite{LSY}. 
	
	Feng, Soffer, Wu and Yao proved local dispersive bounds when $n>2m$ in \cite{soffernew} assuming that $|V(x)|\les (1+|x|)^{-\delta}$ for some $\delta>n$.  Global dispersive bounds are established as a consequence of the first and third author's work on the $L^p$ boundedness of the wave operators provided $\delta>n+3$ when $n$ is odd, \cite{EGWaveOp}, and $\delta>n+4$  for $n$ even, \cite{EGWaveOp2}.  When $n\geq 4m$ some measure of smoothness is needed to ensure global bounds hold, see \cite{EGWaveOp,EGGcounter}.  In all cases, there is a substantial gap to the scaling-critical decay, which roughly corresponds to $\delta>2m$.  We also note that while earlier results require bounded potentials, our results allow for local singularities and establish the uniform global bounds.

	The paper is organized as follows.  First, in Section~\ref{sec:facts}, we collect necessary facts about the higher order Schr\"odinger  resolvent operators and Fourier transforms of related functions.  In Section~\ref{sec:Wiener} we recall the framework for an operator-valued Wiener theorem that we wish to use to establish the dispersive estimates.   In Section~\ref{sec:set up} we prove  Theorem~\ref{thm:main2} which includes  Theorem~\ref{thm:main} as a special case.   Finally, in Section~\ref{sec:wienerproof} we prove Proposition~\ref{prop:inverses} to ensure that the Wiener algebra machinery applies to the higher order Schr\"odinger operators and as a consequence establish facts about the spectral theory of $H$ with $V\in \Kato$.

	\section{Properties of the Resolvent}\label{sec:facts}
	
	Our analysis relies on a careful analysis of resolvent operators.  In this section we collect some needed facts about the resolvent operators and use them to prove results used in the proof of Theorem~\ref{thm:main2}.  Our work relies on understanding properties of the perturbed resolvents $\mR_V^\pm(z)=((-\Delta)^m+V-z)^{-1}$ where the `$+$' and `$-$', denote the limiting values of the resolvent as $z$ approaches $[0,\infty)$ from above and below respectively.
	
	We have the splitting identity for $z\in\C\setminus[0,\infty)$, (c.f. \cite{soffernew})
	\be\label{eqn:Resol}
	\mR_0(z)(x,y):=((-\Delta)^m -z)^{-1}(x,y)=\frac{1}{ mz^{1-\frac1m} }
	\sum_{\ell=0}^{m-1} \omega_\ell R_0 ( \omega_\ell z^{\frac1m})(x,y)
	\ee
	where $\omega_\ell=\exp(i2\pi \ell/m)$ are the $m^{th}$ roots of unity, $R_0(z)=(-\Delta-z)^{-1}$ is the usual ($2^{nd}$ order) Schr\"odinger resolvent. 
	Using the change of variables $z= \lambda^{2m}$ with $\lambda$ restricted to the sector in the complex plane with $0<\arg(\lambda)<\pi/m$,
	\be\label{eqn:Resolv}
	\mR_0( \lambda^{2m})(x,y):=((-\Delta)^m -\lambda^{2m})^{-1}(x,y)=\frac{1}{ m\lambda^{2m-2}}
	\sum_{\ell=0}^{m-1} \omega_\ell R_0 ( \omega_\ell  \lambda^2)(x,y).
	\ee
	By the well-known Bessel function expansions, for $n\geq 3$ odd we have
	\be\label{eqn:R0 explicit0}
	R_0(z^2)(x,y)=\frac{e^{iz |x-y|}}{|x-y|^{n-2}} \sum_{j=0}^{\frac{n-3}{2}} c_{n,j} |x-y|^j z^{j}, \qquad \Im(z)>0.
	\ee   
	Even dimensions are more complicated due to the appearance of logarithmic terms.  We need detailed information on the limiting resolvent operators $\mR_V^\pm(\lambda^{2m})$ as $\lambda$ approaches the positive half line.   
	
	In what follows we need to extend the resolvent to $\lambda \in \R$.  Recall that $R_0(z^2)$ has a logarithmic branch point at the origin.  By taking a branch cut on $\{re^{-i\pi/2m}:r\geq 0\}$ we can guarantee that ${\mR_0}(\lambda )(x,y)$ extends continuously on $\lambda\in\R$.
	However, on the negative half-line, the exponentials experience exponential growth in $\lambda$.  Using the splitting identity, \eqref{eqn:Resolv}, we insert $\eta(\lambda |x-y|)$, a smooth cut-off to the interval $(-1,\infty)$, on the exponentially decaying terms of the resolvent to form:  
	\be\label{breveR}
	\breve{\mR_0^\pm}(\lambda )(x,y)=\frac{1}{ m\lambda^{2m-2}}\bigg[ R_0^\pm(\lambda^2)(x,y)+
	\sum_{\ell=1}^{m-1} \eta(\lambda|x-y|) \omega_\ell R_0 ( \omega_\ell  \lambda^2)(x,y)\bigg].
	\ee

	We combine resolvent representations in \cite{EGWaveOp} with asymptotic expansions in \cite{soffernew} for the following representation of the resolvent.  We write $\la x\ra:=(1+|x|^2)^{\frac12}$ and $f(r)=\widetilde O(\la r\ra^{\beta})$ to denote that $|\partial_r^k f(r)|\leq \la r\ra^{\beta-k}$ for each $k=0,1,2,\dots$.
	
	\begin{prop}\label{prop:resolv G}
		In dimensions $2m<n<4m$, we have the representation
		$$
		\mR_0^\pm(\lambda^{2m})(x,y)=\frac{G^\pm(\lambda|x-y|)}{|x-y|^{n-2m}}=\frac{e^{\pm i \lambda |x-y|} F^\pm(\lambda |x-y|)}{|x-y|^{n-2m}}.
		$$
		When $n$ is odd, $G^\pm, F^\pm \in C^\infty(\R)$, and on $(-K,\infty)$  for any $K<\infty$, they satisfy 
		$|(F^\pm)^{(N)}(r)|\les \la r\ra^{\frac{n+1}{2}-2m -N}$,  $|(G^\pm)^{(N)}(r)|\les \la r\ra^{\frac{n+1}{2}-2m }$, $N=0,1,2,\dots.$
		In addition, the Taylor expansion of $G^\pm$ at zero is of the form $a_0+c_\pm r^{n-2m}+\cdots$.  
		Moreover, when $n=4m-1$,  $F^\pm(r)= c  +\widetilde O(\la r\ra^{ -1})$, where $c$ is the same for $+$ and $-$ cases.
		
		When $n$ is even,  $G^\pm, F^\pm$ are $ C^\infty$ away from the origin,  and the bounds above hold on $(-K,\infty)\setminus(-1,1)$. On $(-1,1)$,  we have $G^\pm, F^\pm\in C^{2m-1}$, and the bounds above hold for $N=0,1,2,\dots, 2m-1$. In addition, we   have the expansion
		$$
		G^\pm (r)= a_0+d_0^\pm r^{n-2m} + \cdots + d_{2m-2}^\pm r^{2m-2} + d_{2m} r^{2m}\log(|r|) + E^\pm(r),  
		$$
		where $E^\pm(r)$ is $C^{2m}$ and satisfies,  $|\partial_r^N E_\pm(r)|\les r^{2m-N}$ for each $0\leq N\leq 2m$.
	\end{prop} 
	\begin{proof}
		We drop  the $\pm$ signs in the proof and consider only the $+$ case, the $-$ case follows easily. 
		This representation is a refinement of Lemmas~3.2 and 6.2
		in \cite{EGWaveOp}, which considered odd and even dimensions respectively, adapted to considering the dispersive estimates.  Here we consider  the function $G$ without  extracting the oscillatory phase and its derivatives at zero in more detail.  The bounds on the derivatives of $F$ are in \cite{EGWaveOp}.  For $G$, the claims follow by writing $G(r)=e^{ ir}F(r)$ and applying the bounds for $F$.  By combining the phase with $F$, we note that the large $r$ decay of $G$ is slower for derivatives.   
		For odd $n$, we take advantage of the explicit formulas   \eqref{eqn:R0 explicit0} and the  splitting identity to see
		\begin{align*}
			G(\lambda|x-y|) &= |x-y|^{n-2m}\mR_0(\lambda^{2m})(x,y) =  \frac{|x-y|^{n-2m}}{ m\lambda^{2m-2}}\big[R_0 ( \lambda^2)(x,y)+
			\sum_{\ell=1}^{m-1} \omega_\ell R_0 ( \omega_\ell \lambda^2)(x,y)\big]\\
			&=  \frac{e^{i\lambda|x- y|}}{ m(\lambda|x-y|)^{2m-2}}\big[P_{\frac{n-3}2} (\lambda|x-y|)   +
			\sum_{\ell=1}^{m-1} \omega_\ell e^{i(\omega_\ell^{\f12}-1)\lambda|x-y|} P_{\frac{n-3}2} (\omega_\ell^{\f12} \lambda |x-y| )   \big] 
		\end{align*}
		Here $P_{k}(s)$ indicates a polynomial of degree $k$ in $s$.
		This also implies the final claim for $F$ when $n=4m-1$. Note that $c$ is independent of the $\pm$ signs because $\frac{n-3}2=2m-2$ is even. 
		
		The claim on the Taylor expansion at zero can be obtained by re-writing the expansions in Proposition 2.4 of \cite{soffernew}, see (2.6), one has:
		$$
		G(r)= a_0+\sum_{j=0}^m c_j r^{n-2m+2j}+E(r)
		$$
		where $E(r)$ has a zero of order $2m+1$ in $r$ and $a_0,c_j$ are constants.  Since $2m<n<4m$, we see that the $G(r)$ has a leading order  behavior   $a_0+cr^{n-2m}$. 
		
		For even $n$, the expansion on $(-1,1)$ follows from equation (2.7) in \cite{soffernew}.
	\end{proof}
	By construction, $\breve{\mR_0}(\lambda)$   have the same properties as $\mR_0(\lambda^{2m})$ in Proposition~\ref{prop:resolv G} and Lemma~\ref{lem:GFFourier} below  for $\lambda\in\R$.  Therefore, by a slight abuse of notation, we will use $G$ and $F$ to denote
	$$
	G(\lambda |x-y|)=|x-y|^{n-2m}\breve{\mR_0}(\lambda )(x,y),\,\,\,\,\,F(\lambda |x-y|)=|x-y|^{n-2m}e^{-i\lambda|x-y|}\breve{\mR_0}(\lambda )(x,y).
	$$
	Using Proposition~\ref{prop:resolv G} we obtain the following bounds on the Fourier transform of the function $G$ and its derivatives.  In the lemma below for $\gamma\notin \mathbb Z$ we define $r^{\gamma}$ by analytic continuation with a branch cut on the negative imaginary axis.	
	\begin{lemma}\label{lem:GFFourier}
		For any $k=1,2,...$ if $n$ is odd, or  $k=1,2,\ldots, \tfrac{n}2 $ if $n$ is even, and any $\gamma$ such that $\min(0,k+2m-n)\leq \gamma< 2m-\frac{n+1}2 $,  we have
		$$
		\mF\big( r^\gamma G^{(k)}(r) \big)\in L^1(\R).
		$$
		Moreover, the Fourier transform is a finite measure if $\gamma=2m-\frac{n+1}2 $.  
	\end{lemma}
	\begin{proof}  
		We first consider the case when $n$ is odd.  Note that for $|r|\les 1$, $G^{(k)}(r)=  r^{-\min(0,k+2m-n)}   E_k(r),$ for some $E_k\in C^2$. Therefore, the Fourier transform of
		$r^\gamma G^{(k)}(r) \chi(r)$ is in $L^1$ by the second part of Lemma~\ref{lem:fhat L1}. 
		For $|r|\gtrsim 1$, we have 
		$$
		r^\gamma G^{(k)}(r) = e^{ir} \sum_{j=0}^k c_{j,k}  r^\gamma  F^{(j)}(r) .
		$$
		By Proposition~\ref{prop:resolv G}, we have 
		$$
		\Big|\Big(\sum_{j=0}^k c_{j,k} r^\gamma  F^{(j)}(r) \Big)^{(N)}\Big|\les \la r\ra^{\frac{n+1}{2}-2m -N+\gamma}.
		$$
		When $\gamma<2m-\frac{n+1}{2}$,   using the first part of Lemma~\ref{lem:fhat L1} below, the Fourier transform of this function is in $L^1$.   
		When $\gamma=2m-\frac{n+1}{2}$, the leading contribution from the sum above comes from the $j=0$ term. We get $c e^{ir} + e^{ir}\widetilde O(\la r \ra^{-1})$. Hence, the Fourier transform is a sum of an  $L^1$ function and a Dirac-$\delta$. 
		
		When $n$ is even, the proof above is valid for $|r|\gtrsim 1$ when $\gamma<2m-\frac{n+1}{2}$. The case of $\gamma=2m-\frac{n+1}{2}$ requires slightly more care.  Here we recall the representation of the second order resolvent in terms of Bessel functions,
		$$
		R_0^+(\lambda^2)(x,y)=\frac{i}{4} \bigg( \frac{\lambda}{2\pi |x-y|} \bigg)^{\frac{n-2}{2}} H^{(1)}_{\frac{n-2}{2}}(\lambda |x-y|)
		$$
		By the splitting identity, \eqref{eqn:Resolv}, we have
		$$
		|x-y|^{2m-n}\mR_0^+(\lambda^{2m})(x,y)
		=\frac{C_n}{(\lambda|x-y|)^{2m-\frac{n+2}{2}}} \sum_{\ell=0}^{m-1} \omega_\ell H^{(1)}_{\frac{n-2}{2}}( \omega_\ell^{\frac12}  \lambda|x-y| )
		$$
		Hence we may write
		$$
		G(r)=\frac{C_n}{r^{2m-\frac{n+2}{2}}} \sum_{\ell=0}^{m-1} \omega_\ell H^{(1)}_{\frac{n-2}{2}}(\omega_\ell^{\frac12}  r )
		$$
		The Bessel function has an expansion of the form (c.f. p.364 of \cite{AS})) as $|z|\to \infty$
		$$
		H_{\frac{n-2}{2}}^{(1)}(z)=e^{i(z-\frac{n+1}{4})}\sqrt{\frac{2}{\pi z}}\, \bigg(1+ \mathcal E(z)\bigg)
		$$
		where $|\mathcal E^{(N)}(z)|\les z^{-\frac{3}{2}-N}$ for all $N=0,1,2,\dots$.  Above we consider the same branch used in defining $r^\gamma$.  By the exponential decay when $\ell\neq 0$, it suffices to consider the case of $\omega_0=1$:
		$$
		r^\gamma G_0^{(k)}(r)=\sum_{j=0}^k \frac{c_{n,k}}{r^{j-\frac12}}  e^{i(r-\frac{n+1}{4})}\sqrt{\frac{2}{\pi r}} \bigg(1+ \mathcal E(r)\bigg)
		$$
		As in the $n$ odd case, the leading contribution is from the $j=0$ term.  Again, we have $ce^{ir}+\widetilde O(\la r\ra^{-1}) $ which suffices to prove the claim.
		
		For $|r|\les 1$, we have 
		$$
		r^\gamma G^{(k)}(r) =  c_0 r^{-\min(0,k+2m-n)+\gamma} + \cdots + \tilde c_{2m}  r^{2m-k+\gamma} + c_{2m+1} r^{2m-k+\gamma}\log(|r|) + r^\gamma E^{(k)}(r)  .
		$$
		The terms before the logarithm are controlled by the second part of Lemma~\ref{lem:fhat L1} by the choice of $\gamma$.  We now turn to the logarithmic term.
		In the worst case (when    $\gamma=\min(0,k+2m-n)$), the logarithmic term is $r^{\min(2m-k,4m-n)}\log(|r|)$, which is at worst  $r \log(|r|)$ (when $k=\tfrac{n}2=2m-1$), which is also controlled by Lemma~\ref{lem:fhat L1}. Finally we note that $|E^{(k)}(r)| \les r^{2m-k}$ and may be differentiated at least one time with $|E^{(k+1)}(r)| \les r^{2m-k-1}$ by Proposition~\ref{prop:resolv G} noting that $k+1\leq \frac{n}{2}+1\leq 2m$. Therefore, $|r^\gamma E^{(k)}(r)| \les r^{\min(2m-k,4m-n)} $, and its derivative is bounded and again Lemma~\ref{lem:fhat L1} applies. 
	\end{proof}
	
	\begin{lemma}\label{lem:fhat L1}
		If $f$ is  $C^1$, supported on $\R\setminus (-1,1)$,   and $|f^{(j)}(r)|\les |r|^{-j-\epsilon}$ for some $\epsilon>0$ and $j=0,1 $, then $\widehat{f}$ is in $L^1(\R)$. 
		
		Furthermore, if $f$ is continuous and compactly supported in $(-2,2)$ and its distributional derivative is a function satisfying  $|f'(r)|\les |r|^{-\epsilon}$ for some $\epsilon<1$, then $\widehat{f}$ is in $L^1(\R)$.
	\end{lemma}

	\begin{proof} 
		Since $f \in \dot H^1$, $\widehat{f}(\rho) =|\rho|^{-1}g(\rho)$, where $g\in  L^2$. Therefore $\widehat f$ is integrable on $|\rho|\gtrsim 1$.  It remains to prove that $\widehat{f}$ is integrable as $|\rho|\to 0$. Note that by an integration by parts (taking $\rho , r >0$)
		$$
		|\widehat f(\rho)| =\Big|\int_1^{\frac1\rho} f(r) e^{-ir\rho} dr+\frac{1}{ie^i\rho} f(1/\rho) -\int_{\frac1\rho}^\infty f^\prime(r) \frac{e^{-ir\rho}}{-i\rho} dr \Big|
		$$ 
		$$
		\les  \rho^{\epsilon-1}+\int_{1}^{\frac1\rho} r^{-\epsilon} dr+\int_{\frac1\rho}^\infty \frac{1}{r^{1+\epsilon} \rho} dr
		\les  \rho^{\epsilon-1}.$$
		This yields the first claim.  
		
		For the second claim, we first note that $\widehat f\in L^2\cap L^\infty$ by the first assumption.  If $\epsilon<\frac12$, then $f\in H^1$ and hence $\widehat f\in L^1$.  We now assume that $\frac{1}{2} \leq \epsilon<1$.  Fix $p$ such that $1<p<\frac{1}{\epsilon}\leq 2$, then $f'\in L^p$.    By Hausdorff-Young, we have $\widehat{f'}\in L^{p'}$.  Further, since  $\widehat{f}$ is also bounded, $\la \xi\ra\widehat{f}(\xi)\in L^{p'}$.  Now, we note
		$$
		\int |\widehat f(\xi)|\, d\xi \leq  \int  |\widehat f(\xi)| \frac{ \la \xi\ra}{\la \xi\ra}\, d\xi \les   \| \widehat f(\xi) \la \xi\ra\|_{p'} \| \la \xi\ra^{-1} \|_{p}<\infty.
		$$
	\end{proof}

	We make some elementary observations about how the Kato norms are related.
	\begin{proposition} \label{prop:KatoMaps}
		For all $x, y, z \in \R^n$, and $0 \leq \alpha \leq \beta \leq \gamma$,
		\[
		\frac{1}{|x-z|^\beta|z-y|^{\gamma +\alpha - \beta}} \les \frac{1}{|x-y|^\alpha}\Big(\frac{1}{|x-z|^\gamma} + \frac{1}{|z-y|^\gamma}\Big).
		\]
		As a consequence, for any $0 \leq \alpha \leq \beta \leq n-2m$ and $V \in \Kato^{n-2m}$,
		\begin{equation} \label{eqn:KatoMaps}
			\Big\| V(\cdot) \int_{\R^n} \frac{f(y)}{|\,\cdot\,-y|^{n-2m + \alpha - \beta}}dy \Big\|_{\Kato^\beta} \les \|V\|_{\Kato^{n-2m}} \|f\|_{\Kato^\alpha}.
		\end{equation}
	\end{proposition}
	
	The first inequality can be verified by examining the regions where $|x-z| < \frac12|x-y|$, where $|z-y| < \frac12|x-y|$ and where $|x-z|, |z-y| \geq \frac12|x-y|$. The norm bound then follows by applying Fubini's theorem and using the pointwise inequality to dominate $\int_{\R^n} \frac{V(z)}{|x-z|^\beta |z-y|^{n-2m+\alpha - \beta}}dz$ by $|x-y|^{-\alpha}$.

	Our goal is to understand the perturbed resolvent operators $\mR_V^\pm$.  We do this by relating the perturbed resolvent to the free resolvents, specifically by applying the resolvent identity, we may write
	\begin{align}\label{eq:resolvent id}
		\mR_V^\pm (\lambda^{2m})&=(I+\mR_0^\pm (\lambda^{2m})V)^{-1}\mR_0^\pm (\lambda^{2m})\\
		&=\mR_0^\pm(\lambda^{2m})(I+V\mR_0^\pm(\lambda^{2m}))^{-1}.\nn
	\end{align}
	To employ the Wiener inversion machinery for the   operators $(I+  \mR_0^\pm(\lambda^{2m})V ) $ and $(I+ V\mR_0^\pm(\lambda^{2m}) ) $, we need another extension of $\mR_0^\pm(\lambda^{2m})$  to the real line in $\lambda$. We cannot use $\breve \mR_0(\lambda)$ here since the spectral assumptions are not necessarily valid for $I+V\breve \mR_0(\lambda)$ when $\lambda<0$.
	To this end, we extend the free  resolvents as follows:
	\begin{align}\label{eq:resolvent extns}
		\widetilde{\mR_0}^\pm(\lambda)=\left\{
		\begin{array}{ll}
			\mR_0^\pm(\lambda^{2m}) & \lambda \geq 0\\
			\mR_0^\mp ((-\lambda)^{2m}) & \lambda<0
		\end{array}
		\right.
	\end{align}
	With this extension the spectral assumptions of Theorem~\ref{thm:main2} can be used to show that $(I+V\widetilde{\mR_0}^\pm(\lambda) ) $ and $(I+\widetilde{\mR_0}^\pm(\lambda)V) $ are invertible for all $\lambda\in \R$.  For notational convenience, we often omit the $\lambda$ dependence of these operators as well as the $\pm$ signs.
	To keep the different extensions used clear, we also introduce the notation 
	\be\label{MrMl}
	M_\ell:=(I+\widetilde{\mR_0}  V)^{-1},\quad\quad M_r:=(I+V\widetilde{\mR_0}  )^{-1}.
	\ee
	To establish the results in Theorems~\ref{thm:main} and \ref{thm:main2}, we show that the inverse operators above are well-defined in an appropriate sense.  The next section explores these operators in depth.

	\section{An Abstract Wiener Theorem} \label{sec:Wiener}
	
	The operator-valued Wiener inversion theorem we will use is an application of the
	one proved in \cite{BG}. In particular, $I + V\widetilde{\mR_0}^\pm(\lambda)$ is a function of $\lambda$ taking values in $\B(L^1(\R^n))$, the set of bounded linear operators on $L^1(\R^n)$. Its Fourier transform with respect to $\lambda$ (and with $\rho$ as the dual variable) will be integrable in the sense described below.
	
	Given a Banach lattice $X$ of functions on $\R^n$, let the space $\U_X$ consist of Borel measures $T$ supported on $\R^{1+2n}$ for which the marginal distribution of $|T|$, formally written as
	\[
	M(T)(x,y) = \int_\R |T(\rho, x, y)|\,d\rho
	\]
	and more precisely defined by
	\[
	M(T)(E) := |T|(\R \times E) \text{ for all Borel subsets $E \subset \R^{2n}$,}
	\]
	defines a bounded integral operator on $X$.  The natural norm on this space is
	\[
	\| T \|_{\U_X} := \| M(T)\|_{\B(X)}.
	\]
	
	With $\mathcal M$ denoting the finite complex-valued measures on $\R$, we note that
	elements of $\U_X$ also act as bounded operators on $X_x \mathcal M_\rho$.  Here $X_x \mathcal M_\rho$ is the compound space of measures $\nu$ on $\R\times\R^n$ for which $M(x)=\|\nu(\cdot, x)\|_{\mathcal M}$  is finite for a.e. $x\in \R^n$ and  belongs to the Banach space $X$.  Elements of $\U_X$ act on this space
	through the formal convolution
	\[
	TF(\rho, x) = \int_{\R^n} \int_\R T(\rho-\sigma, x, y)F(\sigma, y) \,d\sigma dy.
	\]
	This equips $\U_X$ with the structure of a unital Banach Algebra whose identity element $\mathbbm 1$ is given by the measure
	$\delta_{\rho=0} \otimes \delta_{x=y}$.
	
	One may define the Fourier transform of $T \in \U_X$ with respect to the $\rho$
	variable by
	\[
	\widehat{T}(\lambda, x, y) := \int_\R e^{-i\lambda \rho}T(\rho, x, y)\,d\rho.
	\]
	For each $\lambda \in \R$, the kernel $\widehat{T}(\lambda, x,y)$ is dominated pointwise by
	$M(T)$, hence $\|\widehat{T}(\lambda)\|_{\B(X)} \leq \|T\|_{\U_X}$.
	
	Given two distinct Banach lattices $X$ and $Y$ of functions on $\R^n$, we define $\U_{X,Y}$ as in~\cite{BG} to be the set of Borel measures $T$ supported on $\R^{1+2n}$ such that
	$M(T)$ defines a bounded integral operator from $Y$ to $X$. For the same reason as above,
	$\|\widehat{T}(\lambda)\|_{\B(Y,X)} \leq \|T\|_{\U_{X,Y}}$.  It is straightforward to show that the expected algebraic relations, 
	\be\label{ST}
	\|ST\|_{\U_{X,Z}} \leq \|S\|_{\U_{X,Y}} \|T\|_{\U_{Y,Z}} \quad \text{and} \quad
	(\widehat{ST})(\lambda) = \widehat{S}(\lambda)\widehat{T}(\lambda),
	\ee
	both hold in this degree of generality.
	We also write  $\widehat T(\lambda)\in \mF\U_{X,Y}$ if $T\in \U_{X,Y}$.

	A key step in our proof of the dispersive bounds is to show that $(I + V\widetilde{\mR_0}^\pm(\lambda))^{-1}\in \mathcal F\U_{L^1}$, that is it has a Fourier transform in $\U_{L^1}$.  It follows by duality that
	$(I + \widetilde{\mR_0}^\pm(\lambda)V)^{-1}$ belongs to $\mathcal F\U_{L^\infty}$.
	Control of the operator inverses will follow from the following result once we verify that all the conditions are satisfied for $X = L^1(\R^n)$ and $T = \mathcal F(V\widetilde{\mR_0}(\cdot))$.
	
	\begin{theorem}\cite[Proposition 3.1]{BG} \label{thm:Wiener}
		Suppose $T \in \U_{X}$ is such that
		\begin{enumerate}
			\item[a)] For some $N > 0$, $\lim\limits_{\delta \to 0} 
			\norm[T^N(\rho,x,y) - T^N(\rho-\delta,x,y)][\U_{X}] = 0$.
			\item[b)] $\lim\limits_{R \to \infty}
			\norm[\chi_{|\rho| \ge R} T][\U_X] = 0$.
		\end{enumerate}
		If $I + \widehat{T}(\lambda)$ is an invertible element of $\B(X)$ for every
		$\lambda \in \R$, then $\mathbbm 1+ T$ is invertible in 
		$\U_{X}$.
	\end{theorem}
	Our argument to handle the full range of cases $2m < n < 4m$ makes use of some intermediate Kato norms as defined in~\eqref{eq:KatoNorm}. We will need to verify that the conditions of Theorem~\ref{thm:Wiener} are also   satisfied for $X = \Kato^\alpha$, $0 \leq \alpha \leq n-2m$. In particular, in Section~\ref{sec:wienerproof}, we will prove
	\begin{prop}\label{prop:inverses}
		
		Under the conditions of Theorem~\ref{thm:main}, we have that
		\begin{align*}
			&\|   M_r \|_{\mF\U_{\Kato^\alpha}} <\infty \text{ for $0 \leq \alpha \leq n-2m$}, \quad \text{and} \quad
			\| M_\ell \|_{\mF\U_{L^\infty}} <\infty,
		\end{align*}
		where $M_r=(I+V\widetilde{\mR_0})^{-1}$ and $M_\ell=(I+\widetilde{\mR_0}V)^{-1}$.
		
	\end{prop}
	
	The proof of this proposition appears in Section~\ref{sec:wienerproof} below.

	\section{Proof of Theorem~\ref{thm:main2}}\label{sec:set up}

	In this section we use Proposition~\ref{prop:inverses} to prove Theorem~\ref{thm:main2}, from which Theorem~\ref{thm:main} follows as a special case.   By the well-known Stone formula, we may write:
	\begin{align}\label{eq:Stone}
		|H |^{\beta-\frac{n}{2m}}  e^{-itH}P_{ac}(H)f=\frac{1}{2\pi i} \int_0^\infty e^{-it\lambda }  \lambda^{\beta-\frac{n}{2m}}  [\mR_V^+-\mR_V^-](\lambda ) f \, d\lambda.
	\end{align}
	Inserting a smooth, even cut-off function $\chi$ with $\chi(y)=1$ on $[-1,1]$ and $\chi(y)=0$ if $|y|\geq 2$, with $L\geq 1$ along with the convenient change of variables $\lambda \mapsto \lambda^{2m}$ we have  the oscillatory integral:
	\begin{align}\label{eq:Stone useful}
		|H |^{\beta-\frac{n}{2m}} e^{-itH} P_{ac}(H) 	 \chi(H/L)  =
		\frac{m}{ \pi i} \int_0^\infty e^{-it\lambda^{2m} }     \lambda^{2m\beta-n+2m-1} \chi(\lambda/L)   E(\lambda)    \, d\lambda,
		%e^{itH}P_{ac}(H)\chi(H/L)=\frac{m}{\pi i} \int_0^\infty e^{it\lambda^{2m} }   \lambda^{2m-1}  \chi(\lambda/L) E(\lambda) \, d\lambda,
	\end{align}
	where 
	\begin{align}\label{eq:Elambda}
		E(\lambda):= [\mR_V^+-\mR_V^-](\lambda^{2m} ).   
	\end{align}
	Our key observation is 
	\begin{theorem}\label{thm:Elambdak} For any  $k=0,1,2,...,\lfloor \frac{n-1}{2} \rfloor$, 
		a suitable extension  of $\frac{\lambda^k \partial_\lambda^k E(\lambda)}{\lambda^{n-2m}}$ to $\lambda\in \R$ is in $\mF\U_{L^\infty,L^1}$.   
	\end{theorem}

	\begin{proof}   We start with the case $k=0$. Note that by the resolvent identity, we can rewrite $E(\lambda)$ as 
		$$
		E(\lambda)= (I+  \mR_0^+(\lambda^{2m})V )^{-1} [ \mR_0^+-\mR_0^-](\lambda^{2m}) (I+ V\mR_0^-(\lambda^{2m}) )^{-1}.
		$$
		We extend the boundary operators to $\R$ as in \eqref{eq:resolvent extns} introducing the operators $M_r,M_\ell$, \eqref{MrMl}. By Proposition~\ref{prop:inverses}, we have $ M_r\in \mF\U_{L^1}$ and $M_\ell\in \mF\U_{L^\infty}. $  
		
		From the splitting identity, \eqref{eqn:Resolv}, since $\omega_\ell^{\frac12}\lambda^2\notin [0,\infty)$ for $\ell\neq 0$, we see that
		$$
		\frac{[ \mR_0^+-\mR_0^-](\lambda^{2m})(x,y)}{\lambda^{n-2m} }=\frac{[R_0^+-R_0^-](\lambda^2)(x,y)}{\lambda^{n-2}},
		$$	 
		where $R_0^\pm$ is the usual second order resolvent.  By  \cite{GG1} in odd dimensions and Lemma 3.5 in \cite{GG2} for even dimensions, one has that
		\be\label{eqn:resolv diff g}
		\frac{[R_0^+-R_0^-](\lambda^2)(x,y)}{\lambda^{n-2}}=g(\lambda |x-y|)
		\ee
		for some  $g\in C^2$.  (When $n$ is odd $g\in C^\infty$, while $g\in C^{\frac{n}{2}-1}$ when $n$ is even.) Moreover, for $|r|>1$, we can write $g(r)$ as a linear combination  of $e^{\pm ir}f(r)$, with $|f^{(N)}(r)|\les \la r\ra^{\frac{1-n}{2}-N}$. This  implies that $\widehat g\in L^1$, and hence the   Fourier transform of $  \frac{[ \mR_0^+-\mR_0^-](\lambda^{2m})(x,y)}{\lambda^{n-2m} }$ is in $\U_{ L^\infty, L^1 }$.
		The claim for $k=0$ now follows from \eqref{ST}.  
		
		Now we consider $\frac{\lambda \partial_\lambda E(\lambda)}{\lambda^{n-2m}} = 	\frac{\lambda \partial_\lambda [\mR_V^+-\mR_V^-](\lambda^{2m} )}{\lambda^{n-2m }} $.  We restrict our attention to the contribution of $\mR_V^+$, the argument for $\mR_V^-$ follows analogously.  
		By the resolvent identity, \eqref{eq:resolvent id}, we may write $\partial_\lambda \mR_V^+\, (\lambda^{2m})$ in terms of the free resolvents as follows:
		\begin{align} \label{eqn:RVderiv}
			\partial_\lambda \mR_V^+ (\lambda^{2m})=  (I+\mR_0^+(\lambda^{2m})V)^{-1}  \big[\partial_\lambda \mR_0^+(\lambda^{2m})\big] (I+V\mR_0^+(\lambda^{2m}) )^{-1} 
		\end{align}
		We extend the boundary operators to $\R$ as above introducing  $M_r,M_\ell$, \eqref{MrMl}. 	The middle resolvent in the above expression must be handled in a different manner since it's differentiablity to multiple orders at zero is crucial for our analysis. 
		We use the operator  $\breve{\mR_0}(\lambda )(x,y)$ (see  \eqref{breveR})  for this, and hence  extend $\frac{		\lambda \partial_\lambda \mR_V^+ (\lambda^{2m})}{\lambda^{n-2m}}$ to $\R$ as follows  
		$$
		M_\ell \frac{\lambda \partial_\lambda\breve{\mR_0}^+(\lambda )}{\lambda^{n-2m }} M_r  
		=  M_\ell  \frac{ G'(\lambda |x_1-x_2|)}{(\lambda|x_1-x_2|)^{n-2m-1}}M_r.  
		$$
		By Lemma~\ref{lem:GFFourier} with $\gamma=1+2m-n$, the Fourier transform of  $\frac{ G'(r)}{r^{n-2m-1}}$ is a finite measure, and hence the operator in the middle is  in $\mF\U_{ L^\infty, L^1 }$, which again suffices by \eqref{ST}. 
		
		For higher derivatives, we only consider the contribution of  $\mR_V^+$, and we extend the operators to $\R$ as above. The following identities are helpful for differentiating~\eqref{eqn:RVderiv} further:
		\begin{align*}
			\partial_\lambda (I + \mR_0^+(\lambda^{2m})V)^{-1} &= (I + \mR_0^+(\lambda^{2m})V)^{-1} \partial_\lambda\mR_0^+(\lambda^{2m}) (I + V\mR_0^+(\lambda^{2m}))^{-1}V, \\
			\partial_\lambda (I + V\mR_0^+(\lambda^{2m}))^{-1} &= (I + V\mR_0^+(\lambda^{2m}))^{-1} V \partial_\lambda\mR_0^+(\lambda^{2m}) (I + V\mR_0^+(\lambda^{2m}))^{-1}.
		\end{align*}
		With that, $\frac{\lambda^k \partial_\lambda^k \mR_V^+}{\lambda^{n-2m }}$ can be extended to $\R$ as a linear combination of the degenerate term, $M_\ell\frac{\lambda^k \partial_\lambda^k \breve\mR_0^+}{\lambda^{n-2m }}M_r$ where all derivatives act on the inner resolvent, and terms where at least one of $M_\ell$ or $M_r$ are differentiated which are of the form
		\[
		\frac{\lambda^k}{\lambda^{n-2m}} M_\ell  \big[\partial_\lambda^{k_0}\breve{\mR_0} (\lambda )\big]  \Big[ \prod_{i=1}^I M_r  V\partial_\lambda^{k_i}\breve{\mR_0} (\lambda ) \Big] 
		M_r,
		\]
		where $k_0, k_{i}$ are all $\geq 1$ and their sum is $k$.
		
		Proposition~\ref{prop:inverses} states that $M_\ell \in \mF\U_{L^\infty}$ and $M_r \in \mF\U_{\Kato^\alpha}$ for
		$0 \leq \alpha \leq n-2m$. Recall that $L^1(\R^n) = \Kato^0$. We will distribute the powers of $\lambda$ into the product as follows.
		\[
		M_\ell  \frac{\lambda^{\tilde{k}_0}\partial_\lambda^{k_0}\breve{\mR_0} (\lambda )}{\lambda^{n-2m}}  \Big[ \prod_{i=1}^I M_r  V \lambda^{\tilde{k}_i} \partial_\lambda^{k_i}\breve{\mR_0} (\lambda ) \Big] 
		M_r,
		\]
		where $\tilde{k}_i = \min(k_i, \lfloor  2m-\frac{n+1}{2}\rfloor )$ for each $i = 1, 2, \ldots, I$ and $\tilde{k}_0 = k - \sum_i \tilde{k}_i$. 
		
		Note that the kernel of  $\lambda^{\tilde{k}_i} \partial_\lambda^{k_i}\breve{\mR_0} (\lambda)$ has the form
		\[
		\frac{(\lambda |x_1 - x_2|)^{\tilde{k}_i} G^{(k_i)}(\lambda|x_1 - x_2|)}{|x_1 - x_2|^{n-2m - (k_i - \tilde{k}_i)}}
		\]
		Lemma~\ref{lem:GFFourier} implies that the Fourier transform of $r^{\tilde{k}_i}G^{(k_i)}(r)$ is an $L^1$ function or finite measure since $\tilde{k}_i \leq \lfloor 2m - \frac{n+1}{2}\rfloor $ and
		$\tilde{k}_i \geq k_i + 2m - n$. The fact that $k_i \leq \lfloor \frac{n-1}{2}\rfloor$ ensures that the upper and lower bounds are compatible with one another.
		
		As a consequence, Proposition~\ref{prop:KatoMaps} implies that $V \lambda^{\tilde{k}_i} \partial_\lambda^{k_i}\breve{\mR_0} (\lambda)$ belongs to $\mF \U_{\Kato^{\alpha_i}, \Kato^{\alpha_{i-1}}}$ with $\alpha_{i} - \alpha_{i-1} = k_i - \tilde{k}_i$ provided $\alpha_{i-1} \geq 0$ and $\alpha_i \leq n-2m$.  Similarly, Proposition~\ref{prop:inverses} gives us $M_r \in \mF\U_{\Kato^{\alpha_i}}$ for any $\alpha_i$ in that range.
		
		Start with $\alpha_0 = 0$ so that $\Kato^{\alpha_0} = L^1(\R^n)$. This gives a sequence $0 = \alpha_0 \leq \alpha_1 \leq \ldots \leq \alpha_I$ where the maximal value of $\alpha_I$ is achieved when $I =1$ and
		$k_1 = k-1$ is as large as possible.  Thus $\alpha_I \leq \lfloor \frac{n-1}{2}\rfloor -1  - \lfloor 2m - \frac{n+1}{2}\rfloor = n - 2m - 1$ as desired.
		
		Putting all these pieces together with~\eqref{ST} yields
		\[
		\Big[ \prod_{i=1}^I M_r  V \lambda^{\tilde{k}_i} \partial_\lambda^{k_i}\breve{\mR_0} (\lambda )\Big] M_r \in \mF\U_{\Kato^{\alpha_I},L^1} 
		\]
		with $\alpha_I = \sum_i (k_i - \tilde{k}_i)$.
		
		Now we consider the leftmost operator $M_\ell  \frac{\lambda^{\tilde{k}_0}\partial_\lambda^{k_0}\breve{\mR_0} (\lambda )}{\lambda^{n-2m}}$. Recall that $k_0 = k - \sum_i k_i$ and $\tilde{k}_0 = k - \sum_i \tilde{k}_i$. Thus $\tilde{k}_0 - k_0 = \sum_i (k_i - \tilde{k}_i) = \alpha_I$. We have the expression
		\[
		\frac{\lambda^{\tilde{k}_0}\partial_\lambda^{k_0}\breve{\mR_0} (\lambda )}{\lambda^{n-2m}} = \frac{G^{(k_0)}(\lambda |x_1 - x_2|)}{(\lambda |x_1 -x_2|)^{n-2m-\tilde{k}_0}} \frac{1}{|x_1 - x_2|^{\alpha_I}}.
		\]
		We once again check that Lemma~\ref{lem:GFFourier} applies (with $\gamma=\tilde k_0 +2m-n$). Here $\tilde k_0 +2m-n \leq \lfloor 2m-\frac{n+1}{2}\rfloor < 2m-\frac{n+1}{2}$ because $\tilde{k_0} \leq k \leq \lfloor \frac{n-1}{2}\rfloor$. And it is clear that $ \tilde{k}_0+2m-n \geq k_0+2m-n$ because $\tilde{k}_0 \geq k_0$.  The end result is that $\frac{\lambda^{\tilde{k}_0}\partial_\lambda^{k_0}\breve{\mR_0} (\lambda )}{\lambda^{n-2m}} \in \mF\U_{L^\infty, \Kato^{\alpha_I}}$.  Proposition~\ref{prop:inverses} places $M_\ell \in \mF\U_{L^\infty}$. The claim follows from one last application of~\eqref{ST}. Finally we consider the degenerate term $M_\ell  \frac{\lambda^k\partial_\lambda^k \breve{\mR_0} (\lambda )}{\lambda^{n-2m}}   M_r$.  Using Proposition~\ref{prop:resolv G}, for $k=\lfloor\frac{n-1}{2}\rfloor$, we have
		$$
		\frac{\lambda^k\partial_\lambda^k \breve{\mR_0} (\lambda )}{\lambda^{n-2m}}= (\lambda|x-y|)^{2m+k-n} G^{(k)}(\lambda|x-y|).
		$$
		Here $\gamma=2m+k-n=2m-\frac{n+1}{2}$ when $n$ is odd, and $\gamma=2m-\frac{n}{2}-1<2m-\frac{n+1}{2}$ when $n$ is even.
		Lemma~\ref{lem:GFFourier} applies to show that the Fourier transform is a finite measure when $n$ is odd and in $L^1$ when $n$ is even, which along with the mapping of $M_\ell$ and $M_r$ suffices to prove the claim.

	\end{proof}
	
	\begin{rmk}\label{rmk:n even}
		
		In the proof above, consider the case $k=\frac{n}2-1$ when $n$ is even. We claim that every term, except the degenerate one where no derivatives act on $M_{\ell}$ or $M_r$, can handle one more derivative. Namely,   that
		$$
		\frac{ \lambda^{k+1}  }{\lambda^{n-2m}} \partial_\lambda \Big[ M_\ell  \big[\partial_\lambda^{k_0} \mR_0  (\lambda )\big]  \Big[ \prod_{i=1}^I M_r  V\partial_\lambda^{k_i} \mR_0 (\lambda ) \Big]M_r\Big]
		$$
		has an extension to $\R$ which is in $\mF\U_{L^\infty,L^1}$ provided that  $I\geq 1$ (nondegenerate), $k_0, k_{i}$ are all $\geq 1$ and  their sum is $k=\frac{n}2-1$. This is because the derivative will be a linear combination of products of similar form with at least two  factors  $\partial_\lambda^{k_i}\mR_0$, $i=0,1,..,J$, $J\geq 1$, and with $\frac{n}2=\sum_{i=0}^{J} k_i $. For $i=1,2,...,J$, let $\widetilde k_i =\min(k_i,2m-\frac{n+1}2)$ (as opposed to having the floor of the second term in the minimum). And let $\widetilde k_0= \frac{n}2-\sum \widetilde k_i \leq \frac{n-1}2$ since each $\widetilde k_i\geq \frac12$ and there is at least one of them. Therefore we can use Lemma~\ref{lem:GFFourier} with $\gamma=\widetilde k_0+2m-n\leq 2m-\frac{n+1}2$. All other conditions are satisfied as in the proof above. 
		
	\end{rmk}
	
	The following lemma will be useful in the proof of Theorem~\ref{thm:main2}.   
	\begin{lemma}\label{lem:FT nl3m}
		
		For $m\geq 1$ and $-1<\alpha\leq m-1$,  we have the bound
		$$
		\bigg\| \mathcal F^{-1} \bigg( e^{-it\lambda^{2m}} \mathbbm{1}_{\lambda>0} \lambda^{\alpha} \bigg)
		\bigg\|_\infty \les |t|^{-\frac{  \alpha+1}{2m}}.
		$$   Here the inverse Fourier transform is understood in the sense of distributions.
		
	\end{lemma} 
	\begin{proof}
		This follows by scaling from the fact that $e^{-i\lambda^{2m}}\lambda^\alpha \mathbbm{1}_{\lambda>0}$ has a bounded Fourier transform on $\R$ provided that $m\geq 1$ and $-1<\alpha\leq m-1$. To see this, first note that we can take $\lambda\gtrsim 1$ and consider 
		$$
		\int e^{-i\lambda^{2m}+i\lambda \rho}  \lambda^\alpha \chi(\lambda/2^j) d\lambda,  
		$$  
		where $\chi$ is a smooth cutoff for $\lambda\approx 1$ and  $j\in\N$. 
		
		By van der Corput lemma, noting that the magnitude of the second derivative of the phase is $\approx 2^{j(2m-2)}$, we can bound the integral by 
		$$2^{-j(m-1)} \|\partial_\lambda [\lambda^\alpha \chi(\lambda/2^j) ]\|_{L^1}\les 2^{-j(m-1-\alpha)}.$$
		This implies the claim for $\alpha<m-1$ by summing in $j$.  When $\alpha=m-1$ and $|\rho|\not\approx 2^{j(2m-1)}$,   we instead use  the nonstationary phase estimate 
		$2^{j(\alpha-2m+1)}= 2^{-jm}$, which allows us to sum in $j$ for each fixed $\rho$. 
	\end{proof}
	
	We are now ready to prove Theorem~\ref{thm:main2} which includes Theorem~\ref{thm:main} as a special case:
	\begin{proof}[Proof of Theorem~\ref{thm:main2}]
		We divide the integral in \eqref{eq:Stone useful} into two pieces and rewrite it as follows  (ignoring constants):
		\begin{multline*}
			|H |^{\beta-\frac{n}{2m}} e^{-itH} P_{ac}(H) 	 \chi(H/L)  = \int_\R e^{-it\lambda^{2m} }   \mathbbm{1}_{\lambda>0} \lambda^{2m\beta-1} \chi(\lambda |t|^{\frac1{2m}})\chi(\lambda/L) \frac{\widetilde E(\lambda)}{\lambda^{n-2m}}     \, d\lambda \\
			+ \int_0^\infty e^{-it\lambda^{2m} }    \lambda^{2m\beta-1} \widetilde\chi(\lambda |t|^{\frac1{2m}}) \chi(\lambda/L)\frac{  E(\lambda)}{\lambda^{n-2m}}    \, d\lambda =: I+II,
		\end{multline*}
		where $\widetilde E (\lambda)$ is the extension of $E(\lambda)$ given by Theorem~\ref{thm:Elambdak} and $\widetilde\chi=1-\chi$ is a smooth cutoff for the set $|\lambda|\gtrsim 1$. 
		Note that by Fourier multiplication formula and then Theorem~\ref{thm:Elambdak} with $k=0$, we have for all $\beta>0$
		\begin{multline*}
			\|I\|_{L^1\to L^\infty}
			\les \bigg\|\int_{\R} \mathcal F^{-1} \bigg(   e^{-it\lambda^{2m} }   \mathbbm{1}_{\lambda>0} \lambda^{2m\beta-1} \chi(\lambda |t|^{\frac1{2m}}) \chi(\lambda/L)\bigg)(\rho)\, \mathcal F\bigg(\frac{\widetilde E(\lambda)}{\lambda^{n-2m}} \bigg)(\rho)\, d\rho \bigg\|_{L^1\to L^\infty}
			\\
			\les   \Big\|\frac{\widetilde E(\lambda)}{\lambda^{n-2m}} \Big\|_{\mF\U_{L^\infty,L^1}} \,\,  \Big\|\mathcal F^{-1} \bigg(   e^{-it\lambda^{2m} }   \mathbbm{1}_{\lambda>0} \lambda^{2m\beta-1} \chi(\lambda |t|^{\frac1{2m}}) \chi(\lambda/L)\bigg)(\rho)\Big\|_{L^\infty_\rho}
			\\
			\les \Big\|   \lambda^{2m\beta-1} \chi(\lambda |t|^{\frac1{2m}}) \Big\|_{L^1_\lambda}
			\les |t|^{-\beta}.
		\end{multline*}
		Before we consider the second piece, note that $E_{L,t}(\lambda):= \chi(\lambda/L)\widetilde\chi(\lambda |t|^{\frac1{2m}})   E(\lambda) $ satisfies the claim of Theorem~\ref{thm:Elambdak}   by product rule and since for each $j=0,1,..,$ the Fourier transform  of 
		$$
		\lambda^j \partial_\lambda^j \big[ \chi(\lambda/L)  \widetilde \chi(\lambda |t|^{\frac1{2m}}) \big] 
		$$
		is in $L^1$, uniformly in $L$ and $t$. We therefore write
		$$
		II=  \int_0^\infty e^{-it\lambda^{2m} }    \lambda^{2m\beta-1} \frac{  E_{L,t}(\lambda)}{\lambda^{n-2m}} \, d\lambda.
		$$
		Let  $K =  \lfloor\frac{n-1}2\rfloor $ and let $\beta=K+\epsilon$ for some $\epsilon\in (-K, \frac12]$. We integrate by parts  $ K$   times.\footnote{In fact,   if  $\beta\in (k-\frac12, k+\frac12]$ for some $k=0,1,...,K$, then it suffices to integrate by parts $k$ times. In particular, for the proof of Theorem~\ref{thm:main}, we need the claim of Theorem~\ref{thm:Elambdak} only for $k=0,1,2,$ since $\beta=\frac{n}{2m}\in (1,2) $.}  There are no boundary terms due to the support of the cut-offs.  One acts the operator $(\frac{d}{d\lambda} \frac{1}{\lambda^{2m-1}})^{K}$ on the non-oscillatory part of the integrand, leading us to control the contribution of integrals of the form 
		\begin{multline}\label{Kibp}
			\frac{1}{t^K} \sum_{k=0}^K  c_{k} \int_0^\infty e^{-it\lambda^{2m} } \lambda^{2m\beta-1-(2m-1)K-(K-k)}     \frac{ \partial_\lambda^{k} 	E_{L,t}(\lambda)}{\lambda^{n-2m}}    \, d\lambda\\
			=\frac{1}{t^K} \sum_{k=0}^K  c_{k} \int_\R e^{-it\lambda^{2m} } \lambda^{2m\epsilon-1}   \mathbbm{1}_{\lambda>0}     \widetilde	E_{L,t,k}(\lambda)  \, d\lambda,
		\end{multline}
		where $\widetilde E_{L,t,k}$ is the extension of $\frac{ \lambda^k\partial_\lambda^{k} 	E_{L,t}(\lambda)}{\lambda^{n-2m}}  $ given by Theorem~\ref{thm:Elambdak}. 
		
		We estimate the $L^1\to L^\infty$ norm of these integrals using Lemma~\ref{lem:FT nl3m} with $\alpha=m-1 $, Theorem~\ref{thm:Elambdak},  and Lemma~\ref{lem:fhat L1} together with the support of $E_{L,t}$: 
		\begin{multline*}
			\|II\|_{L^1\to L^\infty}
			\les  |t|^{-K} \sum_{k=0}^K    \big\| \widetilde	E_{L,t,k}(\lambda)  \big\|_{\mF\U_{L^\infty,L^1}} \,\,  \,\,  \Big\|\mathcal F^{-1} \bigg(   e^{-it\lambda^{2m} } \lambda^{m-1}  \mathbbm{1}_{\lambda>0}  \bigg)\Big\|_{L^\infty} \\\times   |t|^{\frac{m-2m\epsilon}{2m}} \big\| \big(\lambda |t|^{\frac1{2m}}\big)^{2m\epsilon-m} \widetilde\chi(10 \lambda|t|^{\frac1{2m}})\big\|_{\mF\U_{L^1}}
			\\
			\les |t|^{-K+\frac12-\epsilon} |t|^{-\frac1{2}} 
			\les  |t|^{-\beta}.
		\end{multline*} 
		This yields the full claim when $n$ is odd and the claim for $\beta\leq \frac{n-1}2$  when $n$ is even. The case $n$ even,
		$\frac{n-1}2<\beta\leq \frac{n}2$ requires more care. Fix $n$ even: $2m<n<4m$, and fix $\beta\in(\frac{n-1}2, \frac{n}2]$. It suffices to consider the term $II$, as the bound for $I$ holds for all $\beta>0$.  Integrating by parts $K=\frac{n}2-1$ times as above leads us to \eqref{Kibp}. We first consider the contribution of the terms $k=0,...,K-1$:
		$$
		\frac{1}{t^K} \sum_{k=0}^{K-1}  c_{k} \int_0^\infty e^{-it\lambda^{2m} } \lambda^{2m(\beta-K)-1 }     \frac{ \lambda^k\partial_\lambda^{k} 	E_{L,t}(\lambda)}{\lambda^{n-2m}}    \, d\lambda.
		$$
		We integrate by parts one more time to obtain (the constants $c_k$ are allowed to vary from line to line)
		\begin{multline*}
			\frac{1}{t^{K+1}} \sum_{k=0}^{K}  c_{k} \int_0^\infty e^{-it\lambda^{2m} } \lambda^{2m(\beta-K-1)-1 }     \frac{ \lambda^k\partial_\lambda^{k} 	E_{L,t}(\lambda)}{\lambda^{n-2m}}    \, d\lambda\\
			=\frac{1}{t^{\frac{n}2}} \sum_{k=0}^{\frac{n}2-1}  c_{k} \int_\R e^{-it\lambda^{2m} } \lambda^{2m(\beta-\frac{n}2)-1}   \mathbbm{1}_{\lambda>0}   \widetilde	E_{L,t,k}(\lambda)  \, d\lambda\\
			=\frac{1}{t^{\frac{n}2}} \sum_{k=0}^{\frac{n}2-1}  c_{k} \int_\R e^{-it\lambda^{2m} } \lambda^{m-1}    \mathbbm{1}_{\lambda>0}  |t|^{\frac{n}2-\beta+\frac12}(|t|^{\frac1{2m}}\lambda)^{2m(\beta-\frac{n}2)-m}  \widetilde \chi (10 \lambda |t|^{\frac1{2m}}) \widetilde	E_{L,t,k}(\lambda)  \, d\lambda.
		\end{multline*}
		Here we use that  $\widetilde \chi (10 y)\widetilde \chi(y)=\widetilde \chi(y)$ to insert the new cut-off.  This allows us to scale away the division by $\lambda$ term and obtain the needed factor of $|t|^{-\frac12}$.
		As above we estimate the contribution of this to $\|II\|_{L^1\to L^\infty}$ by 
		\begin{multline*} 
			|t|^{-\beta+\frac{1}2} \sum_{k=0}^{\frac{n}2-1}    \big\| \widetilde	E_{L,t,k}(\lambda)  \big\|_{\mF\U_{L^\infty,L^1}} \,\,  \,\,  \Big\|\mathcal F^{-1} \bigg(   e^{-it\lambda^{2m} } \lambda^{m-1}  \mathbbm{1}_{\lambda>0}  \bigg)\Big\|_{L^\infty} \\\times  \big\| \big(|t|^{\frac1{2m}}\lambda)^{2m(\beta-\frac{n}2)-m}) \widetilde\chi(10 \lambda|t|^{\frac1{2m}})\big\|_{\mF\U_{L^1}}
			\les |t|^{-\beta+\frac12} |t|^{-\frac1{2}} 
			\les  |t|^{-\beta}.
		\end{multline*} 
		The contribution of the term $k=K=\frac{n}2-1$ is more delicate:
		$$
		\frac{1}{t^{K}} \int_0^\infty e^{-it\lambda^{2m} } \lambda^{2m(\beta-K)-1 }     \frac{ \lambda^{K} \partial_\lambda^{K} 	E_{L,t}(\lambda)}{\lambda^{n-2m}}    \, d\lambda.
		$$
		First note that by the Remark~\ref{rmk:n even}, all the terms in $\partial_\lambda^{K} 	E_{L,t}(\lambda)$, except the degenerate one,  can be handled by another integration by parts as the terms $k=0,...,K-1$ above. 
		We therefore focus on  the degenerate term:
		$$
		\frac{1}{t^K} \int_0^\infty e^{-it\lambda^{2m} } \lambda^{2m(\beta-K)-1 }   \widetilde\chi(  \lambda|t|^{\frac1{2m}})  \chi(\lambda/L)  \frac{ \lambda^{K}}{\lambda^{n-2m}}   M_\ell  \big[\partial_\lambda^{K} \breve\mR_0  (\lambda )\big]  
		M_r\, d\lambda.
		$$
		Pairing this with $L^1$ normalized test functions $f$ and $g$ we have 
		$$
		\frac{1}{t^K} \int_{\R^{2n}}\int_0^\infty e^{-it\lambda^{2m} } \lambda^{2m(\beta-K)-1 }   \widetilde\chi(  \lambda|t|^{\frac1{2m}})  \chi(\lambda/L)  \frac{ \lambda^{K}}{\lambda^{n-2m}}      \big[\partial_\lambda^{K} \breve\mR_0 (\lambda)(x,y)\big]   \widetilde M(\lambda,x,y) \, d\lambda dx dy,
		$$
		where 
		$$
		\widetilde M(\lambda,x,y) :=
		M_r(\lambda)( f)(y) \overline{M_r(\lambda)(g)(x)}.
		$$
		It follows from the proof of Theorem~\ref{thm:Elambdak} (the case $k=1$), that the Fourier transform of a suitable extension of $ \lambda \partial_\lambda \mathcal M (\lambda,x,y) $ is in $L^1_{\rho,x,y}$ when $2m<n< 4m-2$ and $n$ even.  In this range $rG'(r)$ has a well-behaved Fourier transform by Lemma~\ref{lem:GFFourier}. Therefore, in the discussion below, we will treat the derivative of this term as division by $\lambda$.

		We write the kernel of the operator $\partial_\lambda^{K} \breve\mR_0(\lambda)$  as 
		$$e^{i\lambda|x-y|} \frac{F_h(\lambda|x-y|)}{ |x-y|^{n-2m-K}}  + \frac{F_\ell(\lambda|x-y|)}{ |x-y|^{n-2m-K}} . $$
		Here $F_h$ is supported on $\lambda|x-y|\gtrsim 1$ and $F_\ell$ on the complement.  Note that, ignoring the scaling factor $|x-y|$,  $F_\ell$ is the low energy part  $|r|<1$ of $G^{(K)}(r)$ in Lemma~\ref{lem:GFFourier}. Also note that the upper bound on $\gamma$ in the lemma is due to the high energy part. Therefore we can differentiate $F_\ell$ one more time in $\lambda$ as the terms $k=0,1,..,K-1$ above.
		
		It remains to consider the contribution of the first summand (high energy part). Ignoring the spatial integrals, we write the $\lambda$ integral as
		$$
		\frac{1}{t^K} \int_0^\infty e^{-it\lambda^{2m} +i\lambda |x-y|} \lambda^{2m(\beta-K)-1 }   \widetilde\chi(  \lambda|t|^{\frac1{2m}})  \chi(\lambda/L)  \frac{ \lambda^{K}}{\lambda^{n-2m}}    \frac{F_h(\lambda|x-y|)}{ |x-y|^{n-2m-K}}
		\widetilde M (\lambda,x,y) \, d\lambda.
		$$
		Note that the phase has a critical point at $\lambda_0=\big(\frac{|x-y|}{2m t}\big)^{\frac1{2m-1}}$. Let $\phi(\eta)$ be a smooth cutoff for the set $|\eta|\approx 1$ and let $\widetilde\phi=1-\phi$.
		Consider the contribution of $\phi(\lambda/\lambda_0)$:
		$$
		\frac{1}{t^K} \int_0^\infty e^{-it\lambda^{2m}}  \lambda^{2m(\beta-K)-1 }   \widetilde\chi(  \lambda|t|^{\frac1{2m}})  \chi(\lambda/L)     \phi(\lambda/\lambda_0)     \frac{F_h(\lambda|x-y|)e^{i\lambda |x-y|}}{ (\lambda |x-y|)^{n-2m-K}}
		\widetilde M (\lambda,x,y) \, d\lambda.
		$$
		Note that
		$$
		\lambda^{2m(\beta-K)-1}=\lambda^{m-1} \lambda^{(2m-1)(\beta-K-\frac12)} \lambda^{\beta-K-\frac12} = \lambda^{m-1}\Big(\frac{\lambda}{|\lambda_0|}\Big)^{(2m-1)(\beta-K-\frac12)} \Big(\frac{|x-y| \lambda}{2m|t|}\Big)^{\beta-K-\frac12}  .
		$$
		Therefore, denoting $\phi(\eta)|\eta|^{(2m-1)(\beta-K-\frac12)}$ by $\Phi(\eta)$, we can rewrite the integral above as
		$$
		\frac{1}{|t|^{\beta-\frac12}} \int_0^\infty e^{-it\lambda^{2m}} \lambda^{m-1}    \widetilde\chi(  \lambda|t|^{\frac1{2m}})      \chi(\lambda/L)   \Phi(\tfrac{\lambda}{\lambda_0}) \frac{F_h(\lambda|x-y|)e^{i\lambda |x-y|}}{ (\lambda |x-y|)^{n-2m-\beta+\frac12}}
		\widetilde M (\lambda,x,y) \, d\lambda.
		$$ 
		Note that $-(n-2m-\beta+\frac12)\leq 2m-\frac{n+1}{2}$. Therefore, the Fourier transform of $\frac{ F_h(r)e^{ir}}{r^{n-2m-\beta+\frac12}}$ (suitably extended to $\R$ when equality holds) is a finite measure. Similarly for the Fourier transforms of $\Phi, \widetilde\chi$, and $\chi$, and the total measures are independent of the $L^\infty$ scaling factors. Therefore, extending to $\R$ and applying Fourier multiplication formula as above,  we estimate the $L^1_{x,y}$ norm of this integral by 
		$$
		\les |t|^{-\beta+\frac12} \Big\|\mathcal F^{-1} \bigg(   e^{-it\lambda^{2m} } \lambda^{m-1}  \mathbbm{1}_{\lambda>0}  \bigg)\Big\|_{L^\infty} \| \mF \widetilde M(\lambda,x,y) \|_{L^1_{\rho,x,y}} \les |t|^{-\beta}.
		$$
		It remains to consider the contribution of
		$$
		\frac{1}{t^K} \int_0^\infty e^{-it\lambda^{2m}}  \lambda^{2m(\beta-K)-1 }   \widetilde\chi(  \lambda|t|^{\frac1{2m}})  \chi(\lambda/L)     \widetilde\phi(\lambda/\lambda_0)     \frac{F_h(\lambda|x-y|)e^{i\lambda |x-y|}}{ (\lambda |x-y|)^{n-2m-K}}
		\widetilde M (\lambda,x,y) \, d\lambda.
		$$
		Here, we re-write the integral as
		\begin{multline*}
			\frac{1}{t^K} \int_0^\infty \bigg(\partial_\lambda e^{-it\lambda^{2m}+i\lambda|x-y|} \bigg) \frac{\lambda^{2m(\beta-K)-1 }   \widetilde\chi(  \lambda|t|^{\frac1{2m}})  \chi(\lambda/L)   \widetilde\phi(\lambda/\lambda_0)     F_h(\lambda|x-y|)}{ 2mt(\lambda^{2m-1}-\lambda_0^{2m-1}) (\lambda |x-y|)^{n-2m-K} }
			\widetilde M (\lambda,x,y) \, d\lambda.
		\end{multline*}
		Then, we integrate by parts once.  Since differentiation of all  functions is comparable to division by $\lambda$ in this regime, it suffices to consider the case of
		\begin{multline}\label{eqn:Minve lambda}
			\frac{1}{t^{K+1}} \int_0^\infty   e^{-it\lambda^{2m}}   \frac{\lambda^{2m(\beta-K)-2 }   \widetilde\chi(  \lambda|t|^{\frac1{2m}})  \chi(\lambda/L)   \widetilde\phi(\lambda/\lambda_0) }{(\lambda^{2m-1}-\lambda_0^{2m-1})}    \frac{F_h(\lambda|x-y|)e^{i\lambda|x-y|}}{ (\lambda |x-y|)^{n-2m-K} }
			\lambda \partial_\lambda\widetilde M (\lambda,x,y) \, d\lambda.
		\end{multline}	
		Here we note that
		\begin{multline*}
			\frac{\lambda^{2m(\beta-K)-2 }   \widetilde\chi(  \lambda|t|^{\frac1{2m}})     \widetilde\phi(\lambda/\lambda_0) }{(\lambda^{2m-1}-\lambda_0^{2m-1})}
			=\lambda^{m-1} \frac{\lambda^{2m-1}\widetilde\phi(\lambda/\lambda_0) }{(\lambda^{2m-1}-\lambda_0^{2m-1})} \lambda^{2m(\beta-K-\frac{3}{2})}\widetilde\chi(  \lambda|t|^{\frac1{2m}}) \\
			= |t|^{\frac{3}{2}+K-\beta} \lambda^{m-1} \bigg[\frac{\widetilde\phi(\lambda/\lambda_0) }{(1-(\frac{\lambda_0}{\lambda})^{2m-1})}\bigg] \bigg[(\lambda|t|^{\frac1{2m}})^{2m(\beta-K-\frac{3}{2})}\widetilde\chi(  \lambda|t|^{\frac1{2m}})\bigg]
		\end{multline*}
		By scaling and support considerations, both  $\frac{\widetilde\phi(\lambda/\lambda_0) }{(1-(\frac{\lambda_0}{\lambda})^{2m-1})}$ and $(\lambda|t|^{\frac1{2m}})^{2m(\beta-K-\frac{3}{2})}\widetilde\chi(  \lambda|t|^{\frac1{2m}})$ have Fourier transforms that are finite measures, or $L^1$ functions respectively since $\beta-K-\frac{3}{2}<0$.  Since $K=\frac{n}{2}-1$, we have that $-(n-2m-K)=2m-\frac{n}{2}-1\leq 2m-\frac{n+1}{2}$ so the Fourier transform of the extension of $r^{2m-\frac{n}{2}-1}e^{ir}F_h(r)$ is in $L^1$.  Hence, this contribution of this term to \eqref{eqn:Minve lambda} is controlled by
		$$
		\frac{|t|^{\frac{3}{2}+K-\beta}}{|t|^{K+1}} \Big\|\mathcal F^{-1} \bigg(   e^{-it\lambda^{2m} } \lambda^{m-1}  \mathbbm{1}_{\lambda>0}  \bigg)\Big\|_{L^\infty} \| \mF \lambda \partial_\lambda \widetilde M(\lambda,x,y) \|_{L^1_{\rho,x,y}} \les |t|^{-\beta}.
		$$
		The other terms are handled similarly with $\widetilde M$ in place of its derivative.  We note that there is no danger when the derivative acts on $F_h$ since $\partial_\lambda F_h(\lambda|x-y|)$ behaves like $\lambda^{-1} F_h(\lambda|x-y|)$.  This completes the argument for $n<4m-2$.
		
		We turn to the case of $n=4m-2$.  Here, we need to take slightly more care since $2m-\frac{n+1}{2}=\frac{1}{2}<1$.  We treat the oscillation and the cut-offs exactly the same as above, however we now move $\lambda^{\frac{1}{2}}$ to $F_h$ and $\lambda^{\frac12}$ to the derivative of $\widetilde M(\lambda, x,y)$. We may write
		\begin{multline}\label{eqn:Mtilde deriv}
			\partial_\lambda \widetilde M(\lambda,x,y)  =\partial_\lambda M_r(\lambda)( f)(y) \overline{M_r(\lambda)(g)(x)}+M_r(\lambda)( f)(y) \overline{\partial_\lambda M_r(\lambda)(g)(x)}\\
			=M_r V \mR_0' M_r(f)(y)\overline{M_r(\lambda)(g)(x)}+M_r(\lambda)( f)(y) \overline{\partial_\lambda M_r V \mR_0' M_r(\lambda)(g)(x)}.
		\end{multline}
		Consider the first summand's contribution to $\lambda^{\frac12}\partial\widetilde M(\lambda,x,y)$, the second summand is handled similarly.   We show that the Fourier Transform of
		$$
		\lambda^{\frac12}M_r V \mR_0' M_r =M_r V\bigg( V(\cdot)   \frac{ (\lambda|y_1-\cdot|)^{\frac{1}{2}} G'(\lambda|y_1-\cdot|)}{|y_1-\cdot|^{n-2m+\frac{1}{2}}} \, dy_1  \bigg) M_r  
		$$
		maps $L^1\to \Kato^{\frac12}$.   Lemma~\ref{lem:GFFourier} ensures that the Fourier transform of $r^{\frac{1}{2}}G'(r)$ is a finite measure, the operator with integral kernel $V(\cdot)  |y_1-\cdot|^{2m-n-\frac12}:L^1_{y_1}\to \Kato^{\frac{1}{2}}$ by Proposition~\ref{prop:KatoMaps}.  We now turn to the contribution of
		$$
		\frac{ \lambda^{\frac12} F_h(\lambda|x-y|)e^{i\lambda |x-y|}}{ (\lambda |x-y|)^{n-2m-K}}=\frac{(\lambda|x-y|)^{2m-\frac{n+1}{2}} F_h(\lambda|x-y|)e^{i\lambda |x-y|}}{|x-y|^{\frac12}}.
		$$
		By Lemma~\ref{lem:GFFourier} the Fourier transform of this is a finite measure whose total mass is bounded by $|x-y|^{-\frac12}$.  Noting that if $h\in \Kato^{\frac{1}{2}}$ and $j\in L^1$ we have
		$$
		\int_{\R^{2n}} \frac{h(y) j(x)}{|x-y|^{\frac12}}\, dy\, dx \les \|h\|_{\Kato^{\frac12}}\|j\|_1,
		$$
		Hence, we have that 
		$$
		\bigg\| \mathcal F \bigg[ \frac{F_h(\lambda|x-y|)e^{i\lambda|x-y|}}{ (\lambda |x-y|)^{n-2m-K} }
		\lambda \partial_\lambda\widetilde M (\lambda,x,y) \big] \bigg\|_{L^{1}_{\rho,x,y}}<\infty.
		$$
		This suffices to establish the desired $|t|^{-\beta}$ control of \eqref{eqn:Minve lambda} when $n=4m-2$.
	\end{proof}

	\section{The proof of Proposition~\ref{prop:inverses} }\label{sec:wienerproof}
	In this section we provide the proof of Proposition~\ref{prop:inverses} to complete the proof of Theorems~\ref{thm:main2}.  That is, we establish that the operator  $I+V\widetilde{\mR_0}(\lambda)$  is invertible   on $\Kato^\alpha$   spaces and its inverse  belong to $\mF\U_{\Kato^\alpha}$. The statements about      $I+\widetilde{\mR_0}(\lambda)V$ follow from the $\alpha = 0$ case by duality of $K^0 = L^1$ and $L^\infty$.

	When $0 \leq \alpha < n-2m$, the proof of this follows quickly from Theorem~\ref{thm:Wiener} and the Lemmas proven below.  Finally, the $\alpha = n-2m$ case is deduced as a corollary.
	To apply the result of Theorem~\ref{thm:Wiener}, we first study the invertibility of $I+V\widetilde{\mR_0}(\lambda)$.  Here we adapt the proofs given for $m=1$ and $n=3$ in \cite{BG} in two ways: to when $2m<n<4m$, and as operators on Kato spaces.
	
	\begin{lemma}\label{prop:L1inv}
		
		Under the assumptions of Theorem~\ref{thm:main},  the operators $I+V\widetilde{\mR_0^\pm}(\lambda)$ are invertible in $\mathcal B(\Kato^\alpha)$ for all $\lambda \in \mathbb R$ and $0 \leq \alpha \leq n-2m$. 
		
	\end{lemma}
	
	The proof of this relies on showing that if $I+V\widetilde{\mR_0^\pm}(\lambda)$ is not invertible, it is equivalent to having a resonance or eigenvalue at energy $\lambda$.  To prove this, we rely on the following result.
	
	\begin{lemma}\label{lem:compact lemma}
		
		$V\widetilde{\mR_0^\pm}(\lambda)$ is a compact operator on $ \Kato^\alpha $ for each $\lambda\in \mathbb R$ and $0 \leq \alpha \leq n-2m$.
		
	\end{lemma}
	
	\begin{proof}
		
		We first consider the case $\alpha = 0$, where $\Kato^\alpha = L^1(\R^n)$. Since $V$ is in the norm closure of smooth compactly supported functions, it suffices to show the result for $V\in C_c^\infty$.  We note that for each $\lambda$, $V\widetilde{\mR_0^\pm}(\lambda)f$ is supported on the support of $V$.  Then, we note
		$$
		(I -\Delta)V\widetilde{\mR_0^\pm}(\lambda)=V\widetilde{\mR_0^\pm}(\lambda)-V(\Delta \widetilde{\mR_0^\pm}(\lambda))f-2\nabla V\cdot \nabla \widetilde{\mR_0^\pm}(\lambda)f-(\Delta V) \widetilde{\mR_0^\pm}(\lambda)f.
		$$
		For any fixed $\lambda\in \R$, using \eqref{eq:resolvent extns} and Proposition~\ref{prop:resolv G}, we have that
		\begin{multline*}
			\|(I -\Delta)V\widetilde{\mR_0^\pm}(\lambda)f(\cdot)\|_{L^1}\\
			\les \int_{\R^{2n}}\bigg(\frac{|V(x)|}{|x-y|^{n-2m+2}}+\frac{|\nabla V(x)|}{|x-y|^{n-2m+1}}+\frac{|\Delta V(x)|}{|x-y|^{n-2m}}\bigg) |f(y)| \, dx\, dy\\
			\leq C_V \int_{\R^n} |f(y)| \int_{\text{supp}(V)} \frac{dx}{|x-y|^{n-2m+2}}dy \leq C \|f\|_1
		\end{multline*}
		Here we use that $n-2m+2<n$ since $m>1$, so that the $x$ integrand is locally integrable.  Thus, for each $\lambda\in \mathbb R$ we have that $V\widetilde{\mR_0^\pm}(\lambda):L^1\to (1-\Delta)^{-1}L^1$ with fixed support inside the support of $V$, hence is compact on $L^1$.
		
		The argument is essentially the same when we consider $V\widetilde{\mR_0^\pm}(\lambda)$ as an operator in $\B(\Kato^\alpha)$, $0 \leq \alpha \leq n-2m$.  Following the above calculations step by step yields the result for $p<\frac{n}{n-2m+2}$
		\begin{multline*}
			\|(I -\Delta)V\widetilde{\mR_0^\pm}(\lambda)f(\cdot)\|_{L^p}\\ \les  \Big\|\int_{\R^{n}}\bigg(\frac{|V(x)|}{|x-y|^{n-2m+2}}+\frac{|\nabla V(x)|}{|x-y|^{n-2m+1}}+\frac{|\Delta V(x)|}{|x-y|^{n-2m}}\bigg)  |f(y)|  \, dy \Big\|_{L^p} \les \|f\|_{\Kato^\alpha}.
		\end{multline*}
		In the last inequality we considered the cases $|x-y|>1$ and $ |x-y|<1$ separately. In the former case  we replaced $|x-y|^{n-2m}$ with $|x-y|^\alpha$. In the latter we put the $L^p$ norm inside by Minkowski integral inequality noting that all singularities belong to $L^p$. By Rellich's theorem, $V\widetilde{\mR_0^\pm}(\lambda)f$ belongs to a compact set in $L^q(\text{supp}(V))$ for $q<\frac{n}{n-2m}$. By Holder, it belongs to a compact set in $\Kato^\alpha$ for $0\leq \alpha<2m$, since $n-2m<2m$, this suffices. 
	\end{proof}
	
	We now prove Lemma~\ref{prop:L1inv}.
	
	\begin{proof}[Proof of Lemma~\ref{prop:L1inv}]
		
		To apply the Fredholm alternative and complete the invertibilty argument, we need to show that a non-trivial solution to $(I+V\widetilde{\mR_0^\pm}(\lambda))\phi=0$ with $\phi\in \Kato^\alpha$ may not exist under the conditions of Theorem~\ref{thm:main}.
		
		First consider the case $\alpha = 0$, where $\phi \in L^1(\R^n)$.
		Assume there is a $\phi\in L^1$ that solves $\phi+V\widetilde{\mR_0}(\lambda)\phi=0$.  Without loss of generality, we take $\lambda\geq 0$ and $\phi$ solves $\phi+ V\mR_0^+(\lambda^{2m}) \phi=0$ by \eqref{eq:resolvent extns}.  If $\lambda<0$, the argument follows through by replacing $\mR_0^+$ with $\mR_0^{-}$.  This implies that $\psi :=\mR_0^+\phi$ solves the equation $\psi+\mR_0^+V\psi=0$.  By Proposition~\ref{prop:resolv G}, for each $\lambda$, the kernel of $\mR_0^+$ is pointwise dominated by the fractional integral operator $I_{2m}$.     
		
		We note that $I_{2m}$ also maps $L^1(\R^n)$ into $L^{2,\frac{n}2-2m-}(\R^n)$.  The mapping bound follows by duality noting that for $2m<n<4m$
		$$
		\sup_y	\int_{\R^n} \frac{1}{|x-y|^{2n-4m} \la x\ra^{4m-n+}} dx \les 1.
		$$

		This means that $\psi = \mR_0^+ \phi$ belongs to $L^{2, \frac{n}2-2m-}(\R^n)$ and is a distributional solution to $H\psi=\lambda \psi$, hence it is a  resonance.  The assumed lack of eigenvalues and resonances (when $\lambda=0$) prohibits such a $\psi$. For $\lambda>0$, we need to conclude that $\psi\in L^2$ in order to rule it out by the assumed absence of embedded eigenvalues. We show this with an argument similar to the limiting absorption principle.
		
		Note that since $V$ is real-valued we have that $\la \mR_0^+\phi,V\mR_0^+\phi\ra$ is real-valued.  Further, since $\phi=-V\mR_0^+ \phi$, that makes $\la \mR_0^+\phi,V\mR_0^+\phi\ra=-\la\mR_0^+\phi,\phi\ra$ and causes the imaginary part of $\la\mR_0^+\phi,\phi\ra$ to vanish.    
		Since we may also identify the imaginary part of $\la\mR_0^+(\lambda^2m)\phi,\phi\ra$ as a multiple of $\int_{|\xi|=\lambda} |\widehat\phi(\xi)|^2\, d\xi$, this implies that $\widehat \phi$ vanishes on $\lambda S^{n-1}$. 
		
		Also note that by factoring $|\xi|^{2m}-\lambda^{2m}$ on the Fourier side, one has 
		\be \label{eqn:ResFactor}
		\mR_0(\lambda^{2m}) \phi=  R_0(\lambda^2) \big[\prod_{\ell=1}^{m-1}R_0(\omega_\ell \lambda^2)\big] \phi
		\ee
		The operator in brackets is Fourier multiplication by $(\lambda^{2m-2} + \lambda^{2m-4}|\xi|^2 + \ldots + |\xi|^{2m-2})^{-1}$, which is real-analytic and behaves like $|\xi|^{-(2m-2)}$ for large $\xi$. We can therefore bound the kernel of that operator by $\frac{e^{-c|x-y|}}{|x-y|^{n-2m+2}}$ for some $c>0$.
		This maps $L^1$ to $L^1\cap L^q$ for $q<\frac{n}{n-2m+2}$.
		Furthermore, $\big[\prod_{\ell=1}^{m-1}R_0(\lambda^2\omega_\ell)\big] \phi$ still has vanishing Fourier transform on the sphere $\lambda S^{n-1}$. We observe that $\frac{2n}{n+4}<\frac{n}{n-2m+2}$ when $2m<n<4m$, so that we may apply Theorem~2 in \cite{GoldHelm} to conclude that $\psi = \mR_0^+(\lambda^{2m})\phi$ is in $L^2(\R^n)$.

		Now consider the general case $\phi \in \Kato^\alpha$. Split $V = V_1 + V_2$, where $V_1 \in C_0$ and $\|V_2\|_\Kato < \eps$ for a small $\eps > 0$ to be determined in a moment.
		Then $\phi +V_2\mR_0(\lambda^{2m})\phi = -V_1\mR_0(\lambda^{2m})\phi$. The pointwise bounds for the free resolvent in Proposition~\ref{prop:resolv G} allow an estimate
		\[
		\|V_1 \mR_0\phi\|_{L^1} \les \iint_{\R^{2n}} \frac{|V_1(x)|}{|x-y|^{n-2m}} |\phi(y)| dx dy \les \int_{\R^n} \frac{|\phi(y)|}{(1+|y|)^{n-2m}} dy < \infty.
		\]
		The constant depends on the size and support of $V_1$, however what is important is that $V_1\mR_0\phi \in L^1$.  Then the $\alpha = 0$ case of Proposition~\ref{prop:KatoMaps} implies that for $\eps$ sufficiently small, $I + V_2\mR_0$ will be a small perturbation of the identity in $\B(L^1)$. Thus
		\[
		\phi = -(I + V_2\mR_0)^{-1}V_1\mR_0\phi \in L^1,
		\]
		which cannot occur without a resonance or eigenvalue if $\lambda = 0$, or an embedded eigenvalue if $\lambda >0$, by the argument just above.
	\end{proof}
	The proof of Lemma~\ref{prop:L1inv} also yields the following corollary.
	
	\begin{corollary}\label{cor:no res}
		
		Under the assumptions of Theorem~\ref{thm:main}, the operator $H=(-\Delta)^m+V$ has no resonances in $(0,\infty)$.
		
	\end{corollary}

	It remains to check that the operators $\mF^{-1}(V\widetilde{\mR_0})$ satisfy the conditions a) and b) of Theorem~\ref{thm:Wiener}.  Note that 
	\be\label{eqn:T defn}
	[V\widetilde{\mR_0}](\lambda,x,y)= \frac{V(x)}{|x-y|^{n-2m}} e^{i\lambda |x-y|} \widetilde F(\lambda|x-y|),  
	\ee
	where $\widetilde F(r)=F^{+}(r)$ for $r\geq 0$ and $F^{-}(-r)$ for $r<0$. An application of Proposition~\ref{prop:KatoMaps} shows that that the integral kernel $\frac{V(x)}{|x-y|^{n-2m}}$ defines a bounded operator on $\Kato^\alpha$ for any $\alpha$ in the range $0 \leq \alpha \leq n-2m$.

	We first prove that $\mF \widetilde F \in L^1(\R) $ when $2m<n<4m-1$, and that it is a finite measure (a Dirac-$\delta$ plus an $L^1$  function) when $n=4m-1$. Let $\chi$ be a smooth cutoff for $[-1,1]$. Note that, by Proposition~\ref{prop:resolv G}, the compactly supported function $\widetilde F(r)\chi(r)$ is  continuous and its derivative is bounded, therefore the second claim in Lemma~\ref{lem:fhat L1} applies.  When $2m<n<4m-1$, the $j^{th}$ derivative of the function $\widetilde F(r)(1-\chi(r))$    decays at least as fast as  $\la r\ra^{-1/2-j}$ , $j=0,1$.  Therefore, its Fourier transform is in $L^1$ by the first claim in Lemma~\ref{lem:fhat L1}.
	
	The case $n=4m-1$ also follows from Lemma~\ref{lem:fhat L1} since $\widetilde F(r)-c$ decays like $\frac1{|r|} $ and derivatives decay faster as above. Therefore $\mF(\widetilde {F}) -c \delta_0 \in L^1$.   The lemma below completes the proof of Proposition~\ref{prop:inverses} when $0\leq \alpha<n-2m$.  The case of $\alpha=n-2m$ will be deduced afterwards with a different argument.
	
	\begin{lemma}\label{lem:wiener check}
		
		Under the conditions of Theorem~\ref{thm:main}, the operator $T:=\widehat{ V\widetilde{\mR_0}}$
		satisfies conditions a) and b) in Theorem~\ref{thm:Wiener} with respect to spaces
		$\U_{\Kato^\alpha}$ provided  $0 \leq \alpha < n-2m$.
		
	\end{lemma}
	
	\begin{proof}
		We first consider the case $2m < n < 4m-1$. By \eqref{eqn:T defn}, we have
		$\widehat{V\widetilde{\mR_0}}(\rho, x, y) = \frac{V(x)}{|x-y|^{n-2m}} \frac{1}{|x-y|}\widehat{\widetilde F}\big(\frac{\rho}{|x-y|} - 1\big)$. Integrating the absolute value with respect to $\rho$ yields
		\[
		M(\widehat{V\widetilde{\mR_0}})(x,y) = \frac{|V(x)|}{|x-y|^{n-2m}} \|\widehat{\widetilde F}\|_{L^1}.
		\]
		Then Proposition~\ref{prop:KatoMaps} with $\alpha = \beta$ indicates that
		$
		\|M(\widehat{V\widetilde{\mR_0}})\|_{B(\Kato^\alpha)} \les \|V\|_{\Kato^{n-2m}} \|\widehat{\widetilde F}\|_{L^1}$  for all  $0 \leq \alpha \leq n-2m$.  In other words,
		\begin{equation} \label{eq:U1VK}
			\|\widehat{V\widetilde{\mR_0}} \|_{\U_{\Kato^\alpha}} \les \|V\|_{\Kato^{n-2m}} \|\widehat{\widetilde F}\|_{L^1}.
		\end{equation}
		
		We now check the conditions a, b in the statement of Theorem~\ref{thm:Wiener} when $2m<n<4m-1$. 
		By a similar calculation as above, we have (with $N=1$)
		$$  
		\norm[[\widehat{ V\widetilde{\mR_0}}](\rho,x,y) - [\widehat{ V\widetilde{\mR_0}}](\rho-\delta,x,y)][\U_{\Kato^\alpha}] \leq  \|V\|_{\mathcal K^{n-2m}} \big\| \widehat{\widetilde F} (\cdot)-\widehat{\widetilde F}(\cdot-\delta)\big\|_{L^1}\to 0  
		$$
		as $\delta$ goes to zero by norm continuity in $L^1$. The condition b follows similarly. Take $V \in \Kato$, the $\Kato$-norm closure of compactly supported continuous functions, $C_0$. Using the bound \eqref{eq:U1VK}, it suffices to prove the claim assuming that $V\in C_0$.  We have   
		\[
		M(\chi_{|\rho|\geq R} \widehat{V\widetilde{\mR_0}})(x,y) = \frac{|V(x)|}{|x-y|^{n-2m}} \int_{|\rho|>R/|x-y|} |\widehat{\widetilde F}(\rho-1)|  d\rho.
		\]
		Everywhere inside the domain $|\rho| > R/|x-y|$, at least one of $|\rho|$ or $|x-y|$ must be greater than $\sqrt{R}$. Thus
		\[
		M(\chi_{|\rho|\geq R} \widehat{V\widetilde{\mR_0}})(x,y)  \leq \frac{|V(x)|}{|x-y|^{n-2m}} \int_{|\rho|>\sqrt{R}} |\widehat{\widetilde F}(\rho-1)|  d\rho +
		\frac{|V(x)| \chi_{|x-y| > \sqrt{R}}}{|x-y|^{n-2m}} \|\widehat{\widetilde F}\|_{L^1}.
		\]
		
		The first term on the right defines a bounded operator on $\Kato^\alpha$ by Proposition~\ref{prop:KatoMaps}, with norm controlled by $\|V\|_{\Kato} \| \widehat{\widetilde F}(\rho - 1)\|_{L^1(|\rho| > \sqrt{R})}$. To determine the $B(\Kato^\alpha)$ norm of the second term, we observe that
		\[
		\int_{\R^n} \frac{|V(x)|\chi_{|x-y|>\sqrt{R}}}{|z-x|^\alpha|x-y|^{n-2m}}dx \leq R^{\frac{\alpha + 2m - n}{2}} \int_{\R^n} \frac{|V(x)|}{|z-x|^\alpha|x-y|^\alpha}dx  \les \frac{R^{\frac{\alpha + 2m-n}{2}}\|V\|_{\Kato^\alpha}}{|z-y|^\alpha},
		\]
		hence
		\[
		\sup_z \int_{\R^n} \frac{|V(x)|\chi_{|x-y|>\sqrt{R}}}{|z-x|^\alpha|x-y|^{n-2m}} f(y) dx dy
		\les  R^{\frac{\alpha + 2m-n}{2}} \|V\|_{\Kato^\alpha}\, \sup_z \int_{\R^n} \frac{|f(y)|}{|z-y|^\alpha} dy .
		\]
		
		Combining the two terms yields
		\[
		\|\chi_{|\rho| \geq R}\widehat{V\widetilde{\mR_0}}\|_{\U_{K^\alpha}}
		\les \|V\|_\Kato \|\widehat{\widetilde F}(\rho - 1)\|_{L^1(|\rho| > \sqrt{R})} + R^{\frac{\alpha + 2m - n}{2}}\|V\|_{\Kato^\alpha},
		\]
		which decreases to zero as $R \to \infty$ since $\alpha < n - 2m$.

		We now turn to the case $n = 4m-1$. It still suffices to verify conditions a and b for all potentials $V \in C_0$ and use \eqref{eq:U1VK} to extend the property to $V \in \Kato$.  The argument for condition b is largely unchanged when $n = 4m - 1$, even though $\widehat{\widetilde F}$ now contains a point mass  $c \delta_0$ in addition to a function in $L^1(\R)$. The essential properties in that argument were
		\[
		\lim_{R \to \infty} \int_{|\rho| > \sqrt{R}} |\widehat{\widetilde F}(\rho - 1)| d\rho = 0
		\]  
		and that $\widehat{\widetilde F}$ has finite total mass, both of which are still true.

		It remains to check condition a when $n=4m-1$ assuming that $V\in C_0$. This condition does not hold for $N=1$ because the total mass of measures $\delta(\,\cdot\,) - \delta(\,\cdot\, - \eps)$ remains large as $\eps \to 0$.  We will show that it holds for $N=3$ instead.
		
		We may assume that $V$ is supported in the ball of radius $M$ centered at the origin.  It follows from the boundedness of $\widetilde F(r)$ in \eqref{eqn:T defn}, which is guaranteed by Proposition~\ref{prop:resolv G}, that $\widetilde{\mR_0}(\lambda, x, y)$ is dominated poinwise by $|x-y|^{2m-n}$ uniformly in $\lambda$. It will be useful to note that for $\sigma \geq 0$, 
		\[
		\int_{|x| < M} \frac{\la x\ra^{2\sigma}}{|x-y|^{2n- 4m}}\,dx \les \frac{1}{\la y\ra^{2n -4m}}
		\]
		with constants depending on $M$ and $\sigma$. We are taking advantage of the fact that $4m > n$ here. Then since $V$ is a bounded function,
		\[
		\|V\widetilde{\mR_0}(\lambda, \,\cdot\, , y)\|_{L^{2,\sigma}} \les \frac{1}{\la y \ra^{n-2m}}.
		\]
		This estimate is sufficient to show that $\|V\widetilde{\mR_0}(\lambda)\|_{\B(L^{2,\sigma})} \les 1$ for all $\lambda \in \R$ and any choice of $\sigma \geq 0$. For $|\lambda| \gtrsim 1$ and larger $\sigma$, a stronger estimate is possible. 
		
		We refer back to the factorization of the free resolvent in~\eqref{eqn:ResFactor}. The operator in brackets is $\lambda^{2-2m}$ times Fourier multiplication with a symbol $(1 + (|\xi|/\lambda)^2  + \ldots + (|\xi|/\lambda)^{2m-2})^{-1}$ whose derivatives up to order $\sigma +1$ are uniformly bounded. Hence this is a bounded operator on $L^{2, \sigma}$ with norm controlled by $\lambda^{2-2m}$.  It is a well-known property of the Schr\"odinger resolvent that $R_0(\lambda^2)$ maps $L^{2, \sigma}$ to $L^{2, -\sigma}$ with norm controlled by $\lambda^{-1}$ provided $\sigma > \frac12$.  Multiplication by $V \in C_0$ maps this back to $L^{2, \sigma}$.  Put together, these estimates show that
		\[ 
		\|V\widetilde{\mR_0}(\lambda)f\|_{L^{2, \sigma}} \les \la \lambda\ra^{1-2m} \|f\|_{L^{2,\sigma}}.
		\]
		
		Choose a fixed $\sigma > 2m - \frac{n}{2}$. Pointwise domination of $\widetilde{\mR_0}(\lambda, x, y)$ by $|x-y|^{2m-n}$ also shows that $V\widetilde{\mR_0}(\lambda)$ maps $L^{2,\sigma}$ to $\Kato^\alpha$ with a bound independent of $\lambda$. Applying all these mapping bounds in order shows that
		\be \label{eqn:N=3} 
		\|(V\widetilde{\mR_0})^3(\lambda, \,\cdot\, , y)\|_{\Kato^\alpha} \les \frac{\la \lambda\ra^{1-2m}}{\la y \ra^{n-2m}}.
		\ee
		
		Note that the difference of Fourier transforms can be written
		\[
		\widehat{(V\widetilde{\mR_0})^3}(\rho) - \widehat{(V\widetilde{\mR_0})^3}(\rho-\delta) = \int_\R e^{-i\rho \lambda} (1 - e^{i\delta \lambda})(V\widetilde{\mR_0})^3(\lambda)\, d\lambda.
		\]
		Then~\eqref{eqn:N=3} and the crude bound $|1 - e^{i\delta\lambda}| \leq \delta |\lambda|$ show that
		\[
		\big\| \widehat{(V\widetilde{\mR_0})^3}(\rho, \,\cdot\, ,y) - \widehat{(V\widetilde{\mR_0})^3}(\rho-\delta, \,\cdot\, ,y) \big\|_{\Kato^\alpha} \les \frac{\delta}{\la y\ra^{n-2m}}.
		\]
		The $\Kato^\alpha$ norm is not changed by taking absolute values, so we can use the Minkowski inequality to bound
		\[
		\Big\| \int_\R \chi_{|\rho|>R} \big|\widehat{(V\widetilde{\mR_0})^3}(\rho, \,\cdot\, ,y) - \widehat{(V\widetilde{\mR_0})^3}(\rho-\delta, \,\cdot\, ,y)\big| d\rho  \Big\|_{\Kato^\alpha} \les \frac{\delta R}{\la y\ra^{n-2m}}
		\]
		for any $R < \infty$. We note that $\la y \ra^{2m - n}$ belongs to the dual space of $\Kato^\alpha$ for $0 \leq \alpha \leq n-2m$. Thus
		\[
		\Big\| \chi_{|\rho|>R} \big[\widehat{(V\widetilde{\mR_0})^3}(\rho) - \widehat{(V\widetilde{\mR_0})^3}(\rho-\delta) \big]  \Big\|_{\U_{\Kato^\alpha}} \les \delta R,
		\]
		which vanishes as $\delta \to 0$. It follows that for any $R < \infty$,
		\begin{align*}
			\lim_{\delta \to 0} \|\widehat{(V\widetilde{\mR_0})^3}(\rho - \delta) - \widehat{(V\widetilde{\mR_0})^3}(\rho)\|_{\U_{\Kato^\alpha}}
			&\leq \lim_{\delta \to 0}\|\chi_{|\rho| > R}[\widehat{(V\widetilde{\mR_0})^3}(\rho - \delta) - \widehat{(V\widetilde{\mR_0})^3}(\rho)]\|_{\U_{\Kato^\alpha}} \\
			&\leq 2 \|\chi_{|\rho| > R}\widehat{(V\widetilde{\mR_0})^3}(\rho)\|_{\U_{\Kato^\alpha}}. 
		\end{align*}
		
		We did not directly verify condition b) for the iterated operator $(V\widetilde{\mR_0})^3$, but thanks to the convolution structure in $\rho$,
		$$
		\|\chi_{|\rho| > R}\widehat{(V\widetilde{\mR_0})^3}(\rho)\|_{\U_{\Kato^\alpha}} \les C(V) \|\chi_{|\rho|> R/3} \widehat{V\widetilde{\mR_0}}(\rho)\|_{\U_{\Kato^\alpha}}.
		$$
		Taking $R \to \infty$ and invoking condition b) for $V\widetilde{\mR_0}$ completes the verification of condition a).
	\end{proof}

	Curiously, the edge case $\alpha = n-2m$ cannot be proved as a direct application of Theorem~\ref{thm:Wiener} with any choice of $N > 1$.
	Instead we write
	\[
	(I + V\widetilde{\mR_0})^{-1} = I - V(I+\widetilde{\mR_0}V)^{-1}\widetilde{\mR_0}
	\]
	We already proved that $(I + \widetilde{\mR_0}V)^{-1} \in \U_{L^\infty}$.  It follows from~\eqref{eqn:T defn} and Lemma~\ref{lem:GFFourier} with $\gamma=k=0$ that for any $f\in \Kato^{n-2m}$ we have (for some $g\in L^1(\R)$)
	$$
	\|\widehat{\widetilde{\mR_0}}f\|_{\U_{L^\infty}}\leq \| g(\rho)\|_{L^1_\rho} \sup_{x\in \mathbb R^n} \int_{\R^n} \frac{|f(x)|}{|x-y|^{n-2m}}\, dx \les \|f\|_{\Kato^{n-2m}}.
	$$
	In particular, this shows that 
	$\widehat{\widetilde{\mR_0}} \in \U_{L^\infty, \Kato^{n-2m}}$. Finally, multiplying by $V$ on the left brings us from $L^\infty$ back to $\Kato^{n-2m}$.
	From here we may conclude that $(I + V\widetilde{\mR_0})^{-1} \in \U_{\Kato^{n-2m}}$. This completes the final case in the proof of Proposition~\ref{prop:inverses}.

	\section*{Statements and Declarations}
	On behalf of all authors, the corresponding author states that there is no conflict of interest. The first author was partially supported by the NSF grant  DMS-2154031 and Simons Foundation Grant 634269.  The second author is partially supported by Simons Foundation
	Grant 635369. The third author is partially supported by Simons Foundation
	Grant 511825.
	No datasets were generated or analysed during the current study.

\end{document}